\documentclass[11pt]{amsart}
\usepackage[latin9]{inputenc}
\usepackage[a4paper]{geometry}
\usepackage{xcolor}
\usepackage{pdfcolmk}
\usepackage{enumitem}
\usepackage{amstext}
\usepackage{amsthm}
\usepackage{amssymb}
\PassOptionsToPackage{normalem}{ulem}
\usepackage{ulem}
\usepackage[unicode=true,pdfusetitle,
 bookmarks=true,bookmarksnumbered=false,bookmarksopen=false,
 breaklinks=false,pdfborder={0 0 1},backref=false,colorlinks=false]
 {hyperref}

\makeatletter

\providecolor{lyxadded}{rgb}{0,0,1}
\providecolor{lyxdeleted}{rgb}{1,0,0}

\DeclareRobustCommand{\lyxsout}[1]{\ifx\\#1\else\sout{#1}\fi}

\numberwithin{equation}{section}
\numberwithin{figure}{section}
\newlength{\lyxlabelwidth}      
\theoremstyle{remark}
\newtheorem*{rem*}{\protect\remarkname}
\theoremstyle{plain}
\newtheorem{thm}{\protect\theoremname}[section]
\theoremstyle{definition}
\newtheorem{defn}[thm]{\protect\definitionname}
\theoremstyle{definition}
\newtheorem*{example*}{\protect\examplename}
\theoremstyle{plain}
\newtheorem{prop}[thm]{\protect\propositionname}
\theoremstyle{plain}
\newtheorem{lem}[thm]{\protect\lemmaname}
	\newenvironment{elabeling}[2][]%
	{\settowidth{\lyxlabelwidth}{#2}
		\begin{description}[font=\normalfont,style=sameline,
			leftmargin=\lyxlabelwidth,#1]}
	{\end{description}}
\theoremstyle{plain}
\newtheorem{cor}[thm]{\protect\corollaryname}

\usepackage{amsthm}\usepackage{bm}
\usepackage{verbatim}
\usepackage{mathrsfs}

\makeatletter
\@addtoreset{footnote}{section}
\makeatother

\makeatother

\providecommand{\corollaryname}{Corollary}
\providecommand{\definitionname}{Definition}
\providecommand{\examplename}{Example}
\providecommand{\lemmaname}{Lemma}
\providecommand{\propositionname}{Proposition}
\providecommand{\remarkname}{Remark}
\providecommand{\theoremname}{Theorem}

\begin{document}
\global\long\def\A{\mathcal{A}}%

\global\long\def\B{\mathcal{B}}%

\global\long\def\I{\mathcal{I}}%

\global\long\def\R{\mathbb{R}}%

\global\long\def\Q{\mathbb{Q}}%

\global\long\def\Z{\mathbb{Z}}%

\global\long\def\N{\mathbb{N}}%

\global\long\def\C{\mathbb{C}}%

\global\long\def\Z{\mathbb{Z}}%

\global\long\def\D{\mathcal{D}}%

\global\long\def\U{\mathcal{U}}%

\global\long\def\P{\mathbb{P}}%

\global\long\def\G{\mathbb{G}}%

\global\long\def\Y{\mathcal{Y}}%

\global\long\def\mult{\mathcal{M}}%

\global\long\def\ef{\varphi}%

\global\long\def\primlat{\mathcal{\Z}_{\text{prim}}^{k,d}}%

\global\long\def\primlatof#1{\mathcal{\Z_{\text{prim}}^{\text{\text{#1},d}}}}%

\global\long\def\zprim{\mathcal{\Z}_{prim}^{d}}%

\global\long\def\T{\mathbb{T}}%

\global\long\def\RND{\text{RND}}%

\global\long\def\df{\overset{\text{def}}{=}}%

\title{on the image in the torus of sparse points on dilating analytic curves}
\author{Michael Bersudsky}
\begin{abstract}
It is known that the image in $\R^{2}/\Z^{2}$ of a circle of radius
$\rho$ in the plane becomes equidistributed as $\rho\to\infty$.
We consider the following sparse version of this phenomenon. Starting
from a sequence of radii $\left\{ \rho_{n}\right\} _{n=1}^{\infty}$
which diverges to $\infty$ and an angle $\omega\in\R/\Z,$ we consider
the projection to $\R^{2}/\Z^{2}$ of the $n$'th roots of unity rotated
by angle $\omega$ and dilated by a factor of $\rho_{n}$. We prove
that if $\rho_{n}$ is bounded polynomially in $n$, then the image
of these sparse collections becomes equidistributed, and moreover,
if $\rho_{n}$ grows arbitrarily fast, then we show that equidistribution
holds for almost all $\omega$. Interestingly, we found that for any
angle there is a sequence of radii growing to $\infty$ faster then
any polynomial for which equidistribution fails dramatically. In greater
generality, we prove this type of results for dilations of varying
analytic curves in $\R^{d}$. A novel component of the proof is the
use of the theory of o-minimal structures to control exponential sums.
\end{abstract}

\maketitle

\section{Introduction}

\thispagestyle{empty}

\let\thefootnote\relax\footnotetext{This project has received funding from the European Research Council  (ERC) under the European Union's Horizon 2020 research and innovation programme (grant agreement No. 754475). The author also acknowledges the support of ISF grant 871/17.} 

To put our work in context we first note the following general problem.
Consider a Lie group $G$, let $\Gamma\leq G$ be a lattice and let
$\pi:G\to G/\Gamma$ be the natural projection. Assume that $\gamma_{T}:[0,1]\to G$,
$T\in\R_{>0}$ is a family of curves which expand as $T\to\infty$.
Then a natural question that arrises is what are the weak-{*} limits
of the probability measures on $G/\Gamma$ given by $\mu_{T}(f)\df\int_{0}^{1}f(\pi(\gamma_{T}(s)))ds$
for $f\in C_{c}(G/\Gamma)$, as $T\to\infty$.

The above question was extensively studied in recent years, see\textbf{
}for example the following\textbf{ }(not complete) list \cite{Randol,Nimish_curves_spaceoflattices,Niimish_geodesic_dil_hyp,Bjurkland_fish_dilations,LYang_Israel_jour_2016,KSS18,Khalil2020_curves,PYang_curve_2020}.

In this paper we consider in the Euclidean setting a natural discrete
analogue of the above. Namely, we consider sequences of expanding
curves in $\R^{d}$ and we prove statements concerning the distribution
of the projection to the $d$-torus $\T^{d}\df\R^{d}/\Z^{d}$ of probability
counting measures supported on \emph{discrete subsets }of those curves.
We note that problems similar in flavor to the ones studied in this
paper were considered in the setting of hyperbolic surfaces for translations
of horocycles in \cite{Strom_mark_kron_seq} and more recently in
\cite{Burrin_shap_Yu}.

The main results of our paper are described roughly as follows. We
show that there is a certain threshold for the sparsity of the sampled
discrete points so that if it is not crossed, then their images in
the $d$-torus become equidistributed (Theorem \ref{thm:all omega})
and if this threshold is crossed, then the mentioned equidistribution
might fail (Theorems \ref{lem:failure for rnd of order k} and \ref{lem:failure for RND}).
Nevertheless, we have found that when there is no restriction on the
level of sparsity, the equidistribution is still generic with respect
to a certain perturbation (Theorem \ref{thm:almost all omega}).

\subsection{Preliminaries to the main results}

In the following we discuss the types of curves which we study in
the paper and introduce some notations and conventions.

Using the data of an analytic function $\phi:\left[0,1\right]^{m}\times[0,1]\to\R^{d}$,
a sequence $\left\{ \rho_{n}\right\} _{n=1}^{\infty}\subseteq\R$
which diverges to $\infty$ and a sequence $\left\{ \mathbf{x}_{n}\right\} _{n=1}^{\infty}\subseteq[0,1]^{m}$,
we consider the curves
\begin{equation}
\gamma_{n}(t)\overset{\text{def}}{=}\rho_{n}\phi(\mathbf{x}_{n},t),\ t\in\left[0,1\right].\label{eq:curves}
\end{equation}

Our main goal will be to describe the weak-{*} limits of sequences
of probability counting measures on $\T^{d}\df\R^{d}/\Z^{d}$ of the
form
\begin{equation}
\mu_{n}\overset{\text{def}}{=}\frac{1}{n}\sum_{k=1}^{n}\delta_{\pi\left(\gamma_{n}(k/n)\right)},\label{eq:measures of interest}
\end{equation}
where $\pi:\R^{d}\to\R^{d}/\Z^{d}$ denotes the natural projection.

We will say that $\left\{ \mu_{n}\right\} _{n=1}^{\infty}$ equidistribute
if
\[
\lim_{n\to\infty}\mu_{n}(f)=\int_{\T^{d}}f\ d\mathbf{x},\ \ \forall f\in C\left(\T^{d}\right),
\]
where $d\mathbf{x}$ denotes the normalized Haar measure.
\begin{rem*}
In all cases we consider, the sequence of continuous curves equidistributes,
namely 
\[
\lim_{n\to\infty}\int_{0}^{1}f(\pi(\gamma_{n}(t)))dt=\int_{\T^{d}}f\ d\mathbf{x},\ \forall f\in C(\T^{d}).
\]
In qualitative terms, when the points $\left\{ \gamma_{n}(k/n)\right\} _{k=1}^{n}$
don't become sparse on the continuous curves $\gamma_{n}$ as $n\to\infty$,
then one may deduce the equidistribution of (\ref{eq:measures of interest})
from the equidistribution of the continuous curves, but when the points
$\left\{ \gamma_{n}(k/n)\right\} _{k=1}^{n}$ become very sparse on
the continuous curves $\gamma_{n}$ as $n\to\infty$, then it is no
longer possible to use the continuous equidistribution to study the
discrete one. We note that the sparsity of the discrete points $\left\{ \gamma_{n}(\frac{k}{n})\right\} _{k=1}^{n}$
can be measured by the rate of divergence of ratio $\rho_{n}/n$ to
$\infty$ as $n\to\infty.$
\end{rem*}
We now describe a rationality property of curves which will be used
to give conditions for equidistribution. Given $j\in\N$ and a smooth
curve $\gamma:\left[0,1\right]\to\R^{d}$, we define
\[
\mathbf{\gamma}^{(j)}(t)\overset{\text{def}}{=}\left(\gamma_{1}^{(j)}(t),..,\gamma_{d}^{(j)}(t)\right),\ t\in[0,1],
\]
where $\gamma_{i}^{(j)}$ is the $j$-th derivative of $\gamma_{i}$.
\begin{defn}
\label{def:rnd conditions}We say that a smooth curve $\gamma:\left[0,1\right]\to\R^{d}$
is \emph{rationally non-degenerate of order} $\kappa\in\N\cup\{\infty\}$,
if 
\[
\kappa=\sup\left\{ j\in\N\mid\left\langle \mathbf{h},\gamma^{\left(j\right)}(t)\right\rangle \text{ is a non-zero function in \ensuremath{t}, \ensuremath{\forall}\ensuremath{\ensuremath{\mathbf{h}\in\Z^{d}\smallsetminus\left\{ \mathbf{0}\right\} }}}\right\} ,
\]
 and we say that a \emph{family} of smooth curves $\phi:\left[0,1\right]^{m}\times[0,1]\to\R^{d}$
is rationally non-degenerate of order $\kappa\in\N\cup\left\{ \infty\right\} $
if 
\[
\kappa=\inf\left\{ \text{order of \ensuremath{\phi(\mathbf{x},\cdot)}}\mid\mathbf{x\in}[0,1]^{m}\right\} .
\]
\end{defn}

We will abbreviate the term ``rationally non-degenerate'' by $\RND$
throughout the text.

\subsection{Equidistribution for polynomial sparsity}

The first main result we would like to discuss concerns the conditions
for the equidistribution of the measures of the form (\ref{eq:measures of interest}).
Our theorem states that the higher the non-degeneracy of the curves
is, the sparser we can sample the curves to obtain equidistribution.
If the order of non-degeneracy is $\infty$, then $\rho_{n}$ is allowed
to grow at an arbitrary fixed polynomial rate to assure that the sequence
of measures (\ref{eq:measures of interest}) equidistributes.
\begin{thm}
\label{thm:all omega} Let $\phi:\left[0,1\right]^{m}\times[0,1]\to\R^{d}$
be a family of $\RND$ analytic curves of order $\kappa$, with $2\leq\kappa\leq\infty$.
Assume that $\rho_{n}\to\infty$ such that:
\begin{itemize}
\item $\rho_{n}=o(n^{\kappa})$ if $\kappa<\infty$,
\item $\left\{ \rho_{n}\right\} _{n=1}^{\infty}$ grows polynomially if
$\kappa=\infty$, namely, there exists $l\in\N$ such that $\rho_{n}=o(n^{l})$,
\end{itemize}
and let $\left\{ \mathbf{x}_{n}\right\} _{n=1}^{\infty}\subseteq[0,1]^{m}$.
Then, the sequence of measures given by \emph{(\ref{eq:measures of interest})}
for the curves \emph{(\ref{eq:curves})} equidistributes.
\end{thm}

\begin{rem*}
It is an artifact of our proof that Theorem \ref{thm:all omega} is
stated for $\kappa\geq2$ and not for all $\kappa\in\N$. It seems
to us that Theorem \ref{thm:all omega} with $\kappa=1$ is also true,
yet, since for $\rho_{n}=o(n)$ the points $\left\{ \rho_{n}\phi(\mathbf{x}_{n},\frac{k}{n})\right\} _{k=1}^{n}$
don't get sparse on the curves $\rho_{n}\phi(\mathbf{x}_{n},t)$,
we didn't make the effort to include a proof.
\end{rem*}

We find the following examples to be noteworthy.
\begin{example*}
Let $\psi:\left[0,1\right]^{m}\to\text{GL}{}_{2}(\R)$ be an analytic
map, namely, the entries of the matrices $\psi(\mathbf{x})$ are analytic
functions. Then the family of ellipses 
\[
\phi(\mathbf{x},t)=\left(\cos(2\pi t),\sin(2\pi t)\right)\cdot\psi(\mathbf{x}),
\]
 is a $\RND$ family of analytic curves of order $\infty$ to which
we may apply Theorem \ref{thm:all omega}.
\end{example*}
\begin{example*}
The following example shows that the way expanding curves are sampled
can have dramatic effects on the equidistribution of (\ref{eq:measures of interest}).

Let $\alpha$ be an irrational number and consider the following two
parameterizations of the line segment $\{(t,\alpha t)\mid t\in[0,1]\}$,
\[
\gamma_{1}(t)=(t,\alpha t),\ \ \gamma_{2}(t)=\left(\sin\left(\frac{\pi}{2}t\right),\alpha\sin\left(\frac{\pi}{2}t\right)\right),\ t\in\left[0,1\right].
\]
The important distinction between the two curves is that $\gamma_{1}$
is $\RND$ of order 1 and $\gamma_{2}$ is $\RND$ of order $\infty$.
Let $\rho_{n}=n^{\kappa}$ for an arbitrary $\kappa\in\N$, then $\left\{ \pi\left(\rho_{n}\gamma_{1}(k/n)\right)\right\} _{k=1}^{n}$
\emph{will} \emph{not} equidistribute since the first coordinate is
zero modulo one. On the other-hand, Theorem \ref{thm:all omega} implies
that the points $\left\{ \pi\left(\rho_{n}\gamma_{2}(k/n)\right)\right\} _{k=1}^{n}$
will equidistribute.
\end{example*}

\subsection{Counter examples}

Our goal in the following is to discuss the possibility of failure
of equidistribution when the conditions on $\left\{ \rho_{n}\right\} _{n=1}^{\infty}$
given in Theorem \ref{thm:all omega} are not met.

\subsubsection{The case of $\protect\RND$ curves of finite order}

We now consider dilations of a single curve
\[
\gamma_{n}(t)=\rho_{n}\gamma(t),
\]
where $\gamma:\left[0,1\right]\to\R^{d}$ is analytic. The content
of Theorem \ref{lem:failure for rnd of order k} below is to show
that for a $\RND$ curve of order $\kappa<\infty$, the condition
$\rho_{n}=o(n^{\kappa})$ of Theorem \ref{thm:all omega} is rather
sharp.
\begin{thm}
\label{lem:failure for rnd of order k} Assume that $\gamma:\left[0,1\right]\to\R^{d}$
is a $\RND$ analytic curve of order $\kappa<\infty$. Then there
exists a sequence $\left\{ \rho_{n}\right\} _{n=1}^{\infty}$ which
satisfies $n^{\kappa}\leq\rho_{n}\leq n^{\left(\kappa+1\right)^{2}},\ \forall n\in\N$,
so that $\left\{ \mu_{n}\right\} _{n=1}^{\infty}$ given by \emph{(\ref{eq:measures of interest})}
for the curves $\gamma_{n}=\rho_{n}\gamma$ will not equidistribute.
\end{thm}

\begin{example*}
There are $\RND$ analytic curves of order $\kappa$ such that $\left\{ \mu_{n}\right\} _{n=1}^{\infty}$
will not equidistribute for $\rho_{n}=n^{\kappa}$. Indeed, consider
\[
\gamma(t)=\left(t^{\kappa},t^{\kappa+1}\right),
\]
then the first coordinate of $n^{\kappa}\gamma(\frac{j}{n})$ is zero
modulo one.
\end{example*}

\subsubsection{The case of $\protect\RND$ curves of order $\infty$}

If we consider an arbitrary family of curves $\phi:\left[0,1\right]^{m}\times\left[0,1\right]\to\R^{d}$,
then the following is true.
\begin{thm}
\label{lem:failure for RND} Let $\phi:[0,1]^{m}\times[0,1]\to\R^{d}$
be a family of curves and fix $\left\{ \mathbf{x}_{n}\right\} _{n=1}^{\infty}\subseteq\left[0,1\right]^{m}.$
Then there exists a sequence $\left\{ \rho_{n}\right\} _{n=1}^{\infty}$
diverging to $\infty$ with 
\begin{equation}
\rho_{n}\ll\left(\left(3.5\right)^{d}\right)^{n},\ \forall n\in\N,\label{eq:upperbound for rho_n for rnd}
\end{equation}
 for which $\left\{ \mu_{n}\right\} _{n=1}^{\infty}$ given by\emph{
(\ref{eq:measures of interest})} for the curves $\gamma_{n}=\rho_{n}\phi(\mathbf{x}_{n},\cdot)$
will not equidistribute.
\end{thm}

If $\phi:\left[0,1\right]^{m}\times\left[0,1\right]\to\R^{d}$ is
a family of $\RND$ analytic curves of order $\infty$, then by Theorem
\ref{thm:all omega} it is necessary that the sequence $\left\{ \rho_{n}\right\} _{n=1}^{\infty}$
given in Lemma \ref{lem:failure for RND} satisfies for all $\kappa\in\N$
that
\[
\lim_{n\to\infty}\frac{n^{\kappa}}{\rho_{n}}=0.
\]

Namely, for $\RND$ curves of order $\infty$, the ``bad dilations''
necessarily exceed polynomial growth, yet there are ``bad dilations''
which can be bounded exponentially.

\subsection{Equidistribution beyond polynomial growth}

To illustrate the content of our last main result (Theorem \ref{thm:almost all omega})
we now discuss a particular example.

Consider for $n\in\N$ the rotated $n$'th roots of unity on the unit
circle
\[
r_{k,\omega}^{(n)}\df\left(\sin\left(\frac{2\pi k}{n}+\omega\right),\cos\left(\frac{2\pi k}{n}+\omega\right)\right),\ k=1,..,n,\ \omega\in\R/\Z.
\]
We denote by $\left\{ \rho_{n}\right\} _{n=1}^{\infty}\subseteq\R$
a sequence diverging to $\infty$ and we consider the sequence of
measures 
\[
\mu_{n,\omega}\df\frac{1}{n}\sum_{k=1}^{n}\delta_{\pi(\rho_{n}r_{k,\omega}^{(n)})}
\]
 for $n\in\N,$ $\omega\in\R/\Z$.

Consider the case that $\rho_{n}$ is bounded polynomially in $n$.
Since for all $\omega\in\R/\Z$ the curve
\[
\gamma_{\omega}(t)=\left(\sin\left(2\pi t+\omega\right),\cos\left(2\pi t+\omega\right)\right),\ t\in[0,1]
\]
is $\RND$ analytic curve of order $\infty$, we obtain by Theorem
\ref{thm:all omega} that for all\emph{ }$\omega\in\R/\Z$ the sequence
of measures $\left\{ \mu_{n,\omega}\right\} _{n=1}^{\infty}$ equidistributes
as $n\to\infty$.

Theorem \ref{thm:almost all omega} below sheds light on the case
that $\left\{ \rho_{n}\right\} _{n=1}^{\infty}$ diverges to infinity
at an arbitrary rate by stating that $\left\{ \mu_{n,\omega}\right\} _{n=1}^{\infty}$
equidistributes \emph{for almost all} $\omega\in\R/\Z$. Note that
Theorem \ref{lem:failure for RND} states that for any fixed $\omega_{0}\in\R/\Z$
there exists a sequence $\left\{ \rho_{n}\right\} _{n=1}^{\infty}$
diverging to $\infty$ such that $\left\{ \mu_{n,\omega_{0}}\right\} _{n=1}^{\infty}$
does not equidistribute.

\subsubsection{Almost all rotations}

We say that a family of closed curves $\ef:\left[0,1\right]^{m}\times\R/\Z\to\R^{d}$
is $\RND$ analytic of order $\infty$, if the lift of $\ef$ defined
by 
\[
\phi(\mathbf{x},t)=\ef(\mathbf{x},t+\Z),\ \left(\mathbf{x},t\right)\in\left[0,1\right]^{m}\times\left[0,1\right],
\]
is $\RND$ analytic family of curves of order $\infty$. Let $\left\{ \mathbf{x}_{n}\right\} _{n=1}^{\infty}\subseteq\left[0,1\right]^{m}$,
$\left\{ \rho_{n}\right\} _{n=1}^{\infty}\subseteq\R$ and $\omega\in\left[0,1\right]$.
We denote 
\begin{equation}
\gamma_{\omega,n}(t)\overset{\text{def}}{=}\rho_{n}\ef(\mathbf{x}_{n},t+\omega+\Z),\ t\in\left[0,1\right],\label{eq:rotated curves}
\end{equation}
and for $\gamma_{\omega,n}$ we define the measure $\mu_{\omega,n}$
as in (\ref{eq:measures of interest}).
\begin{thm}
\label{thm:almost all omega} Assume that $\ef:\left[0,1\right]^{m}\times\R/\Z\to\R^{d}$
is an analytic family of $\RND$ curves of order $\infty$. Fix $\left\{ \mathbf{x}_{n}\right\} _{n=1}^{\infty}\subseteq\left[0,1\right]^{m}$
and $\left\{ \rho_{n}\right\} _{n=1}^{\infty}\subseteq\R$ such that
$\rho_{n}\to\infty$. Then, $\left\{ \mu_{\omega,n}\right\} _{n=1}^{\infty}$
equidistributes, for almost all $\omega\in[0,1]$.
\end{thm}

\subsection{Proof ideas and organization of the paper}

To prove equidistribution we estimate exponential sums. In these estimates,
it is crucial to control the sub-level sets of the amplitude function
(the one that appears as the argument of the exponential). This control
is reflected both in bounding the number of connected components and
the measure of the sub-level sets as illustrated in Proposition \ref{prop:uniformity on the lower bound and complement}.
This connects our discussion to the theory of o-minimal structures,
which allows to prove such estimates in impressive generality by an
argument which we find to be elegant (see Section \ref{sec:Prelimeneries}).
As far as we know, this work is novel in its use of o-minimality to
control exponential sums, yet we note that the use of o-minimality
to control exponential integrals already appears in \cite{Peterzil_strach_o-min-flow}.

The structure of the paper is as follows:
\begin{itemize}
\item In Section \ref{sec:Prelimeneries} we establish the mentioned properties
of sub-level sets, and in Appendix \ref{sec:Basic-notions-in o-minim}
we discuss the facts which we need from the theory of o-minimal structures.
\item In Sections \ref{sec:Proof-of-theorem all omega} and \ref{sec:Proof-of-a.a. omega}
we prove Theorems \ref{thm:all omega} and \ref{thm:almost all omega}
respectively.
\item In Section \ref{sec:On-the-growth} we discuss counter examples for
equidistribution (Theorems \ref{lem:failure for rnd of order k} and
\ref{lem:failure for RND}).
\end{itemize}

\subsubsection*{Notational conventions.}
\begin{itemize}
\item For $n\in\N,$ we set $[n]\overset{\text{def}}{=}\{1,2,..,n\}.$
\item $e(x)\overset{\text{def}}{=}e^{2\pi ix}$.
\item Vectors will be denoted by bold letters, namely $\mathbf{x}$ will
stand for a n-tuple, and by small letters, as $x_{j}$, we denote
the coordinates of $\mathbf{x}$.
\item We denote for $\mathbf{x},\mathbf{y}\in\R^{d}$ by $\left\langle \mathbf{x},\mathbf{y}\right\rangle $
the usual Euclidean inner product, and by $\left\Vert \mathbf{x}\right\Vert $
the Euclidean norm.
\item For a subset $S\subseteq\R^{d}$ we denote by $\left|S\right|$ its
Lebesgue volume and for any set $S$ we denote by $\#S$ the number
of elements in $S$.
\item We will use the notations $\ll,$ $o\left(\cdot\right),$ $O(\cdot),\ $
$\asymp$, as in the book \cite{Anaytic_num_theo} (see introduction
in \cite{Anaytic_num_theo}).
\end{itemize}

\subsubsection*{Acknowledgements}

The origins of this project is in a question raised by Uri Shapira
during exciting discussions with Erez Nesharim, Ren{\'e} R{\"u}hr, Shucheng
Yu and Cheng Zheng, where many interesting ideas brought up and I
thank them for that. I am grateful to Uri Shapira for his ideas, support,
encouragement and remarks throughout my work on this project. I am
very thankful for many discussions with Ren{\'e} R{\"u}hr and Shucheng Yu
which influenced this work. I am indebted to Kobi Peterzil for a wonderful
seminar talk in our department from which I first noticed the connection
between this problem and the theory of o-minimal structures, and for
many discussions and advice that followed. I am in particular very
thankful to Gal Binyamini for showing me the proof of Proposition
\ref{prop:uniformity on the lower bound and complement} and for a
very interesting discussion. I am thankful for fruitful discussions
with Ze{\'e}v Rudnick and Sergei Yakovenko. I am grateful to the referee,
Daniel Goldberg and Yakov Karasik for corrections and suggestions
which significantly improved the manuscript.

\section{\label{sec:Prelimeneries}Preliminaries}

Let $N\in\N$ and assume that $F:\left[0,1\right]^{N}\times\left[0,1\right]\to\R$
is a non-constant analytic function. Let 
\begin{equation}
\Sigma\df\left\{ \mathbf{x}\in\left[0,1\right]^{N}\mid F(\mathbf{x},t)=0,\ \forall t\in[0,1]\right\} ,\label{eq:definition of Sigma}
\end{equation}
and note that $\Sigma$ is a proper closed subset of $\left[0,1\right]^{N}$.
We denote by $\text{dist}(\mathbf{x},\Sigma)$ the usual Euclidean
distance between $\mathbf{x}\in\R^{N}$ and $\Sigma$. Then, for all
$\epsilon>0$ small enough, the subset 
\begin{equation}
\Sigma_{\epsilon}\df\left\{ \mathbf{x}\in\left[0,1\right]^{N}\mid\text{dist}(\mathbf{x},\Sigma)\geq\epsilon\right\} ,\label{eq:definition of Sigma eps}
\end{equation}
is non-empty, and the function $F(\mathbf{x},\cdot)$ is non-zero
for all $\mathbf{x}\in\Sigma_{\epsilon}$. When $\Sigma=\emptyset$
we take the convention that $\Sigma_{\epsilon}=\left[0,1\right]^{N}$
for all $\epsilon>0$.

We define for $\delta>0$ and $\mathbf{x}\in\left[0,1\right]^{N}$
the set

\[
\mathcal{F}_{\mathbf{x},\delta}\overset{\text{def}}{=}\left\{ t\in\left[0,1\right]\mid\left|F(\mathbf{x},t)\right|\geq\delta\right\} .
\]

\begin{prop}
\label{prop:uniformity on the lower bound and complement}Assume that
$F:\left[0,1\right]^{N}\times\left[0,1\right]\to\R$ is a non-constant
analytic function. Then, there exists $\alpha>0$ such that for all
$\epsilon\in(0,1)$ and $\mathbf{x}\in\Sigma_{\epsilon}$ it holds
that $\mathcal{F}_{\mathbf{x},\epsilon^{\alpha}}$ is a union of $O(1)$
closed intervals, say 
\begin{equation}
\mathcal{F}{}_{\mathbf{x},\epsilon^{\alpha}}=I_{1}\cup..\cup I_{m},\ m=O(1),\label{eq:f(x,delta) finite union of intervals}
\end{equation}
and 
\begin{equation}
1-\left|\mathcal{F}_{\mathbf{x},\epsilon^{\alpha}}\right|\ll\epsilon.\label{eq:bound on area of F_x_eps^alph}
\end{equation}
\end{prop}

The proof of Proposition \ref{prop:uniformity on the lower bound and complement}
involves familiarity with the theory of o-minimal structures. We refer
the reader not familiar with o-minimal theory to Appendix \ref{sec:Basic-notions-in o-minim}
where we discuss all the required details needed for the proof below.

We denote by $\R_{an}$ the o-minimal structure expanding the real
field generated by the restricted analytic functions (see Section
\ref{subsec:The-o-minimal-structure R_an}).

By the \emph{uniform bounds on fibers} property of o-minimal structures,
there exists $M=M(F)>0$ such that $\mathcal{F}_{\mathbf{x},\delta}$
is a union of at most $M$ intervals (see Theorem \ref{thm:uniform bound on fibers}
and observe that $\mathcal{F}\df\left\{ (\mathbf{x},\delta,t)\in[0,1]^{N}\times\R_{>0}\times[0,1]\mid\left(F(\mathbf{x},t)\right)^{2}-\delta^{2}\geq0\right\} $
is definable in $\R_{an}$), which proves (\ref{eq:f(x,delta) finite union of intervals}).

The bound (\ref{eq:bound on area of F_x_eps^alph}) is more special,
namely, it is not necessarily true in an arbitrary o-minimal structure.
The essential ingredient we will use below in the proof of (\ref{eq:bound on area of F_x_eps^alph})
is that $\R_{an}$ is \emph{polynomially bounded} (see Definition
\ref{def:polynomially bounded}).

We consider
\begin{equation}
A\df\left\{ \left(\epsilon,\delta\right)\in(0,1]\mid\forall\mathbf{x}\in\Sigma_{\epsilon},\ \forall\xi\in\left[0,1\right]\text{ it holds that }\ (\xi-\frac{\epsilon}{2},\xi+\frac{\epsilon}{2})\cap\mathcal{F}_{\mathbf{x},\delta}\neq\emptyset\right\} .\label{eq:definition of A}
\end{equation}

\begin{lem}
\label{lem:elements in A bound on length}For all $\left(\epsilon,\delta\right)\in A$
and all $\mathbf{x}\in\Sigma_{\epsilon}$ it holds 
\[
1-\left|\mathcal{F}_{\mathbf{x},\delta}\right|\leq\epsilon(M+1),
\]
where $M$ is a uniform bound on the number of intervals comprising
$\mathcal{F}_{\mathbf{x},\delta}$.
\end{lem}

\begin{proof}
Assume not, namely assume that there exists $\left(\epsilon,\delta\right)\in A$
and $\mathbf{x}\in\Sigma_{\epsilon}$ such that 
\[
1-\left|\mathcal{F}_{\mathbf{x},\delta}\right|>\epsilon\left(M+1\right).
\]
Since $\left[0,1\right]\smallsetminus\mathcal{F}_{\mathbf{x},\delta}$
consists of at most $m+1\leq M+1$ intervals, there exists one of
them, say 
\[
\left[0,1\right]\smallsetminus\mathcal{F}_{\mathbf{x},\delta}\supseteq I_{0},
\]
with length $l>\frac{\epsilon\left(M+1\right)}{m+1}\geq\epsilon$.
Then for the center point of $I_{0}$, say $\xi_{0}\in I_{0}$, we
have 
\[
I_{0}\supseteq\left(\xi_{0}-\frac{\epsilon}{2},\xi_{0}+\frac{\epsilon}{2}\right).
\]
Namely, there exists $\xi_{0}\in\left[0,1\right]$ such that $(\xi_{0}-\frac{\epsilon}{2},\xi_{0}+\frac{\epsilon}{2})\cap\mathcal{F}_{\mathbf{x},\delta}=\emptyset$,
which is a contradiction since $\left(\epsilon,\delta\right)\in A$.
\end{proof}
\begin{lem}
\label{lem:Sigma_epsilon lemma}For all $\epsilon>0$ such that $\Sigma_{\epsilon}\neq\emptyset$
there exists $\delta>0$ such that $\left(\epsilon,\delta\right)\in A.$
\end{lem}

\begin{proof}
Fix an arbitrary $\epsilon>0$ with $\Sigma_{\epsilon}\neq\emptyset$.
Assume (for contradiction) that the statement of the lemma is not
true. Then, for all $\delta\in(0,1]$, $\exists\mathbf{x}_{\delta}\in\Sigma_{\epsilon},\ \exists\xi_{\delta}\in\left[0,1\right],$
such that $\left(\xi_{\delta}-\frac{\epsilon}{2},\xi_{\delta}+\frac{\epsilon}{2}\right)\subseteq[0,1]\smallsetminus\mathcal{F}_{\mathbf{x}_{\delta},\delta},$
namely 
\begin{equation}
\left|F(\mathbf{x}_{\delta},t)\right|<\delta,\ \forall t\in\left(\xi_{\delta}-\frac{\epsilon}{2},\xi_{\delta}+\frac{\epsilon}{2}\right).\label{eq:bound on F(X_delta,t) in nbhd of xi_delta}
\end{equation}
By compactness of $\Sigma_{\epsilon}$ and $\left[0,1\right]$ we
obtain a sequence $\{\delta_{n}\}_{n=1}^{\infty}$ such that $\delta_{n}\to0$,
$\mathbf{x}_{\delta_{n}}\to\mathbf{x}_{0}\in\Sigma_{\epsilon}$ and
$\xi_{\delta_{n}}\to\xi_{0}\in[0,1]$, and by (\ref{eq:bound on F(X_delta,t) in nbhd of xi_delta})
we deduce that there is a neighborhood of $\xi_{0}$ in which $F(\mathbf{x}_{0},\cdot)\equiv0$.
Since $F(\mathbf{x}_{0},\cdot)$ is analytic, it vanishes on $\left[0,1\right]$,
which is a contradiction as $\mathbf{x}_{0}\in\Sigma_{\epsilon}$.
\end{proof}
\begin{proof}[Proof of Proposition \ref{prop:uniformity on the lower bound and complement},
bound (\ref{eq:bound on area of F_x_eps^alph})]
 By Lemma \ref{lem:Sigma_epsilon lemma} there exists $\epsilon_{0}>0$
such that $(0,\epsilon_{0})\subseteq\pi_{2,1}(A)$, where $\pi_{2,1}:\R^{2}\to\R$
is the projection to the first coordinate.

By Lemma \ref{lem:A definability} we get that $A$ is definable in
$\R_{an}$, and by the \emph{definable choice theorem }(see Theorem
\ref{thm:choice function}), there is a definable function $\delta\left(\cdot\right):\left(0,\epsilon_{0}\right)\to\left(0,1\right)$
whose graph is in $A$. Since $\R_{an}$ is polynomially bounded,
we obtain by Corollary \ref{cor:polynomially bounded from below}
that there is an $\alpha>0$ such that $\delta\left(\epsilon\right)\geq\epsilon^{\alpha}$
for all $\epsilon\in\left(0,\epsilon'_{0}\right)$ for some $0<\epsilon_{0}'<\epsilon_{0}$.
This completes the proof since
\[
1-\left|\mathcal{F}_{\mathbf{x},\epsilon^{\alpha}}\right|\leq1-\left|\mathcal{F}_{\mathbf{x},\delta(\epsilon)}\right|\underset{\text{Lemma \ref{lem:elements in A bound on length}}}{\leq}(M+1)\epsilon,\ \ \forall\epsilon\in\left(0,\epsilon'_{0}\right).
\]
\end{proof}

\section{\label{sec:Proof-of-theorem all omega}Proof of theorem \ref{thm:all omega}}

Given $\phi:\left[0,1\right]^{m}\times[0,1]\to\R^{d}$, $\mathbf{x}\in[0,1]^{m}$,
$\rho\in\R_{>0}$ and $\mathbf{h}\in\Z^{d}\smallsetminus\left\{ \mathbf{0}\right\} $,
we define
\begin{equation}
f_{n,\mathbf{x},\rho}(t)\overset{\text{def}}{=}\rho\left\langle \mathbf{h},\phi\left(\mathbf{x},\frac{t}{n}\right)\right\rangle ,\ t\in[0,n].\label{eq:def of f_n}
\end{equation}
We shall always assume that $\phi$ is analytic. The main result of
this section is the following.
\begin{prop}
\label{prop:decay of exponental sums} Let $2\leq l\in\N$ and $\mathbf{h}\in\Z^{d}\smallsetminus\{\mathbf{0}\}$
be such that for all $\mathbf{x}\in\left[0,1\right]^{m}$ the function
$F_{l}(\mathbf{x},t)\df\frac{\partial^{l}}{\partial t^{l}}\left\langle \mathbf{h},\phi\left(\mathbf{x},t\right)\right\rangle $
is non-zero in $t$. Assume that $\left\{ \delta_{n}\right\} _{n=1}^{\infty},\ \left\{ \eta_{n}\right\} _{n=1}^{\infty}\subseteq\left[0,1\right)$
are such that $\lim_{n\to\infty}n^{\delta_{n}}=\lim_{n\to\infty}n^{\eta_{n}}=\infty$.
Then there exists a sequence $\left\{ E_{n}\right\} _{n=1}^{\infty}$
converging to zero such that for all $\mathbf{x}\in\left[0,1\right]^{m}$
and $\rho\in[n^{\delta_{n}},n^{l-\eta_{n}}]$ it holds
\begin{equation}
\frac{1}{n}\sum_{k=1}^{n}e\left(f_{n,\mathbf{x},\rho}(k)\right)\ll E_{n},\label{eq:weyls criterion for dilated curves}
\end{equation}
where the implied constant is independent of $\mathbf{x}\in\left[0,1\right]^{m}$
and $\rho\in[n^{\delta_{n}},n^{l-\eta_{n}}]$.
\end{prop}

\begin{proof}[Proof that Proposition \ref{prop:decay of exponental sums} yields
Theorem \ref{thm:all omega}]
 Let $\phi:\left[0,1\right]^{m}\times[0,1]\to\R^{d}$ be a RND analytic
family of order $\kappa\in\N\cup\left\{ \infty\right\} $. Then by
the definition of $\RND$ (see Definition \ref{def:rnd conditions}),
it holds for all\textbf{ $\mathbf{h}\in\Z^{d}\smallsetminus\{\mathbf{0}\}$}
and for all $\mathbf{x}\in\left[0,1\right]^{m}$ that $\frac{\partial^{l}}{\partial t^{l}}\left\langle \mathbf{h},\phi\left(\mathbf{x},t\right)\right\rangle $
is non-zero in $t$ for all $l\leq\kappa$. Fix $\N\ni l\leq\kappa$
and let $\rho_{n}\to\infty$ such that $\rho_{n}=o(n^{l})$. Then
there exist sequences $\left\{ \delta_{n}\right\} _{n=1}^{\infty},\ \left\{ \eta_{n}\right\} _{n=1}^{\infty}\subseteq\left[0,1\right)$
that satisfy $\lim_{n\to\infty}n^{\delta_{n}}=\lim_{n\to\infty}n^{\eta_{n}}=\infty$
for which 
\[
n^{\delta_{n}}\leq\rho_{n}\leq n^{l-\eta_{n}},
\]
 for all large enough $n$. Using the notation (\ref{eq:curves})
we rewrite (\ref{eq:def of f_n}) by
\[
f_{n,\mathbf{x}_{n},\rho_{n}}(t)=\left\langle \mathbf{h},\gamma_{n}\left(\frac{t}{n}\right)\right\rangle ,
\]
where $\mathbf{x}_{n}\in\left[0,1\right]^{m}$. By Weyl's equidistribution
criterion (see e.g. \cite[Chapter 21]{Anaytic_num_theo}) we deduce
that Proposition \ref{prop:decay of exponental sums} implies the
equidistribution of $\left\{ \mu_{n}\right\} _{n=1}^{\infty}$ (defined
in (\ref{eq:measures of interest})).
\end{proof}
Our main tool in the proof of Proposition \ref{prop:decay of exponental sums}
will be the following Van der Corput estimate which we borrow from
\cite[Chapter 8, Theorem 8.20]{Anaytic_num_theo}.
\begin{thm}
\label{thm:van der corput q derivative test} For all $2\leq j\in\N$
there exists a constant $\kappa_{j}>0$ such that $\forall g\in C^{j}(I)$
where $I=\left[a,b\right]$ with $b-a\geq1$ that satisfy 
\[
\eta\leq\left|\frac{d^{j}}{dt^{j}}g(t)\right|\leq\sigma\eta,\ \ \forall t\in I,
\]
with $\eta>0$ and $\sigma\geq1$, it holds
\[
\left|\sum_{k\in I\cap\Z}e(g(k))\right|\leq\kappa_{j}\left(\sigma^{2^{2-j}}\eta^{\tau_{j}}|I|+\eta^{-\tau_{j}}|I|^{1-2^{2-j}}\right),
\]
where $\tau_{j}=\left(2^{j}-2\right)^{-1}$.
\end{thm}

Before applying Theorem \ref{thm:van der corput q derivative test}
to the exponential sums of our interest (see Lemma \ref{prop:estimate of exp sum depending on nu and j}),
we need the following observation which follows from Section \ref{sec:Prelimeneries}.

Fix $2\leq l\in\N$ and $\mathbf{h}\in\Z^{d}\smallsetminus\{\mathbf{0}\}$
such that for all $\mathbf{x}\in\left[0,1\right]^{m}$ the analytic
function $\frac{\partial^{l}}{\partial t^{l}}\left\langle \mathbf{h},\phi\left(\mathbf{x},t\right)\right\rangle $
is non-zero in $t$. Let $\N\ni j\leq l$ and observe that $\frac{\partial^{j}}{\partial t^{j}}\left\langle \mathbf{h},\phi\left(\mathbf{x},t\right)\right\rangle $
is also non-zero in $t$ for all $\mathbf{x}\in\left[0,1\right]^{m}$.
The latter is equivalent to that $\Sigma^{(j)}$ (defined in (\ref{eq:definition of Sigma}))
is empty for the analytic function $F_{j}(\mathbf{x},t)\df\frac{\partial^{j}}{\partial t^{j}}\left\langle \mathbf{h},\phi\left(\mathbf{x},t\right)\right\rangle $,
for all $j\leq l$. Then, by Proposition \ref{prop:uniformity on the lower bound and complement},
we deduce that there exist $M_{j},\alpha_{j}>0$ such that for all
$\mathbf{x}\in\left[0,1\right]^{m}$ and $\epsilon\in(0,1)$ it holds
that

\[
\mathcal{F}_{\mathbf{x},\epsilon^{\alpha_{j}}}^{(j)}\overset{\text{def}}{=}\left\{ t\in\left[0,1\right]\mid\left|\frac{\partial^{j}}{\partial t^{j}}\left\langle \mathbf{h},\phi\left(\mathbf{x},t\right)\right\rangle \right|\geq\epsilon^{\alpha_{j}}\right\} ,
\]
is a union of at most $M_{j}$ intervals, say
\begin{equation}
\mathcal{F}_{\mathbf{x},\epsilon^{\alpha_{j}}}^{(j)}=I_{1,\mathbf{x}}^{(j)}\cup..\cup I_{m_{x},\mathbf{x}}^{(j)},\ m_{\mathbf{x}}\leq M_{j},\label{eq:f_j,n intervals}
\end{equation}
 and 
\begin{equation}
1-\left|\mathcal{F}_{\mathbf{x},\epsilon^{\alpha_{j}}}^{(j)}\right|\ll\epsilon.\label{eq:f_j,n_complemnt small}
\end{equation}

\begin{lem}
\label{prop:estimate of exp sum depending on nu and j} Assume that
$2\leq j\leq l$ and let $n\in\N$. Then for all $\mathbf{x}\in\left[0,1\right]^{m}$
and $\epsilon\in\left(0,1\right)$ we have
\[
\frac{1}{n}\sum_{k=1}^{n}e(f_{n,\mathbf{x},\rho}(k))\ll
\]
\begin{equation}
\epsilon^{\alpha_{j}(\tau_{j}-2^{2-j})}\left(\frac{\rho}{n^{j}}\right)^{\tau_{j}}+\epsilon^{-\tau_{j}\alpha_{j}}\left(\frac{\rho}{n^{j}}\right)^{-\tau_{j}}n^{-2^{2-j}}+\frac{1}{n}+\epsilon.\label{eq:estimate from vdc}
\end{equation}
where $\tau_{j}=\left(2^{j}-2\right)^{-1}$, and the implied constant
is independent of $\mathbf{x},n$ and $\epsilon$.
\end{lem}

\begin{proof}
Fix $2\leq j\leq l$. First, by using (\ref{eq:f_j,n intervals})
and (\ref{eq:f_j,n_complemnt small}), we obtain that 
\[
\#\left\{ 1\leq k\leq n\mid\ \frac{k}{n}\notin\mathcal{F}_{\mathbf{x},\epsilon^{\alpha_{j}}}^{(j)}\right\} \ll n\epsilon,
\]
which implies by the trivial estimate that
\[
\frac{1}{n}\sum_{\begin{array}{c}
_{1\leq k\leq n,\frac{k}{n}\notin\mathcal{F}_{\mathbf{x},\epsilon^{\alpha_{j}}}^{(j)}}\end{array}}e(f_{n,\mathbf{x},\rho}(k))\ll\epsilon.
\]
Next, let
\[
c_{j}\overset{\text{def}}{=}\sup\left\{ \left|\frac{\partial^{j}}{\partial t^{j}}\left\langle \mathbf{h},\phi\left(\mathbf{x},t\right)\right\rangle \right|\mid\mathbf{x}\in\left[0,1\right]^{m},\ t\in\left[0,1\right]\right\} ,
\]
and note that by the chain rule,
\[
\frac{d^{j}}{dt^{j}}f_{n,\mathbf{x},\rho}(t)=\frac{\rho}{n^{j}}\frac{\partial^{j}}{\partial t^{j}}\left\langle \mathbf{h},\phi\left(\mathbf{x},\frac{t}{n}\right)\right\rangle .
\]
Therefore, for all $t\in nI_{i,\mathbf{x}}^{(j)}$ (appearing in (\ref{eq:f_j,n intervals}))
\begin{equation}
\frac{\rho}{n^{j}}\epsilon^{\alpha_{j}}\leq\left|\frac{d^{j}}{dt^{j}}f_{n,\mathbf{x},\rho}(t)\right|\leq\frac{\rho}{n^{j}}c_{j}.\label{eq:sanwich on phi_j}
\end{equation}
We denote 
\[
\eta=\frac{\rho}{n^{j}}\epsilon^{\alpha_{j}},
\]
\[
\sigma=c_{j}\epsilon^{-\alpha_{j}},
\]
and we rewrite (\ref{eq:sanwich on phi_j}) by
\[
\eta\leq\left|\frac{d^{j}}{dt^{j}}f_{n,\mathbf{x},\rho}(t)\right|\leq\sigma\eta,\ \forall t\in nI_{i,\mathbf{x}}^{(j)}.
\]
Assume $\left|nI_{i,\mathbf{x}}^{(j)}\right|\geq1,$ then by Theorem
\ref{thm:van der corput q derivative test}  

\begin{equation}
\begin{aligned}\frac{1}{n}\sum_{\begin{array}{c}
1\leq k\leq n,\ k\in nI_{i,\mathbf{x}}^{(j)}\end{array}}e\left(f_{n,\mathbf{x},\rho}(k)\right)\\
\ll\epsilon^{\alpha_{j}(\tau_{j}-2^{2-j})}\left(\frac{\rho}{n^{j}}\right)^{\tau_{j}} & +\epsilon^{-\tau_{j}\alpha_{j}}\left(\frac{\rho}{n^{j}}\right)^{-\tau_{j}}n^{-2^{2-j}}.
\end{aligned}
\label{eq:estimate by theorem}
\end{equation}

Together with the trivial estimate on the intervals $nI_{i,\mathbf{x}}^{(j)}$
with $\left|nI_{i,\mathbf{x}}^{(j)}\right|<1$, we find that

\[
\begin{aligned}\frac{1}{n}\sum_{\begin{array}{c}
1\leq k\leq n,\ \frac{k}{n}\in\mathcal{F}_{\mathbf{x},\epsilon^{\alpha_{j}}}^{(j)}\end{array}}e\left(f_{n,\mathbf{x},\rho}(k)\right)\\
\ll\epsilon^{\alpha_{j}(\tau_{j}-2^{2-j})} & \left(\frac{\rho}{n^{j}}\right)^{\tau_{j}}+\epsilon^{-\tau_{j}\alpha_{j}}\left(\frac{\rho}{n^{j}}\right)^{-\tau_{j}}n^{-2^{2-j}}+\frac{1}{n}.
\end{aligned}
\]
\end{proof}

\begin{proof}[Proof of Proposition \ref{prop:decay of exponental sums}]
 Fix $2\leq l\in\N$ and let $\mathbf{h}\in\Z^{d}\smallsetminus\{\mathbf{0}\}$
such that for all $\mathbf{x}\in\left[0,1\right]^{m}$ the function
$\frac{\partial^{l}}{\partial t^{l}}\left\langle \mathbf{h},\phi\left(\mathbf{x},t\right)\right\rangle $
is non-zero in $t$. Let $\left\{ \delta_{n}\right\} _{n=1}^{\infty},\ \left\{ \eta_{n}\right\} _{n=1}^{\infty}\subseteq\left[0,1\right)$
be such that $\lim_{n\to\infty}n^{\delta_{n}}=\lim_{n\to\infty}n^{\eta_{n}}=\infty$
and assume that 
\begin{equation}
n^{\delta_{n}}\leq\rho\leq n^{l-\eta_{n}}.\label{eq:sandwich on rho}
\end{equation}
We pick $\lambda\in\R$ such that
\begin{equation}
\rho=n^{\lambda},\label{eq:rho_n as phi_n}
\end{equation}
so that by (\ref{eq:sandwich on rho})
\[
\delta_{n}\leq\lambda\leq l-\eta_{n}.
\]
According to $\lambda\in\left[\delta_{n},l-\eta_{n}\right]$, we define
\begin{equation}
j_{n}(\lambda)\overset{\text{def}}{=}\begin{cases}
2 & \delta_{n}\leq\lambda\leq1,\\
\left\lceil \lambda\right\rceil +1 & 1<\lambda<l-1,\\
l & l-1\leq\lambda\leq l-\eta_{n},
\end{cases}\label{eq:j_n def}
\end{equation}
and 
\begin{equation}
\nu_{n}(\lambda)\overset{\text{def}}{=}\frac{1}{2}\min\left\{ \frac{j_{n}(\lambda)-\lambda}{\alpha_{j_{n}(\lambda)}\left(\frac{2^{2-j_{n}(\lambda)}}{\tau_{j_{n}(\lambda)}}-1\right)},\frac{1}{\alpha_{j_{n}(\lambda)}}\left(\frac{2^{2-j_{n}(\lambda)}}{\tau_{j_{n}(\lambda)}}-\left(j_{n}(\lambda)-\lambda\right)\right)\right\} ,\label{eq:nu definition}
\end{equation}
where $\tau_{j}$ is defined in Theorem \ref{thm:van der corput q derivative test}
and $\alpha_{j}$ is given in Lemma \ref{prop:estimate of exp sum depending on nu and j}.
We would like to plug in 
\begin{equation}
\epsilon_{n}(\lambda)\overset{\text{def}}{=}n^{-\nu_{n}(\lambda)}\label{eq:def of e_n(lambda)}
\end{equation}
into the estimate (\ref{eq:estimate from vdc}). For that to be useful,
we would like first to verify that there exist $\epsilon_{n}$ such
that 
\[
\epsilon_{n}(\lambda)\leq\epsilon_{n},
\]
and $\epsilon_{n}\to0$ (this, by Lemma \ref{prop:estimate of exp sum depending on nu and j},
will yield estimate (\ref{eq:estimate of vdc with T_i(=00005Clambda)})
below). We verify this by estimating from below each of the terms
appearing in the minimum of (\ref{eq:nu definition}). It will be
useful to note that 
\begin{equation}
\frac{2^{2-j}}{\tau_{j}}=4-2^{3-j}.\label{eq:2^(2-j)/tau_j}
\end{equation}
\begin{itemize}
\item The term $\frac{j_{n}(\lambda)-\lambda}{\alpha_{j_{n}(\lambda)}\left(\frac{2^{2-j_{n}(\lambda)}}{\tau_{j_{n}(\lambda)}}-1\right)}$
: An inspection of (\ref{eq:j_n def}) implies that,
\[
\eta_{n}\leq j_{n}(\lambda)-\lambda,\ \forall\lambda\in[\delta_{n},l-\eta_{n}],
\]
and since $j_{n}(\lambda)\geq2$, we deduce by (\ref{eq:2^(2-j)/tau_j})
that $\frac{2^{2-j_{n}(\lambda)}}{\tau_{j_{n}(\lambda)}}-1\leq3$.
Hence 
\begin{equation}
\frac{j_{n}(\lambda)-\lambda}{\alpha_{j_{n}(\lambda)}\left(\frac{2^{2-j_{n}(\lambda)}}{\tau_{j_{n}(\lambda)}}-1\right)}\geq\frac{\eta_{n}}{3\alpha_{j_{n}(\lambda)}}.\label{eq:first term in the minimum}
\end{equation}
\item The term $\frac{1}{\alpha_{j_{n}(\lambda)}}\left(\frac{2^{2-j_{n}(\lambda)}}{\tau_{j_{n}(\lambda)}}-\left(j_{n}(\lambda)-\lambda\right)\right)$
: First assume that $\lambda>1$. Then, $j_{n}(\lambda)\geq3$, and
as a consequence (see (\ref{eq:2^(2-j)/tau_j}))
\[
\frac{2^{2-j_{n}(\lambda)}}{\tau_{j_{n}(\lambda)}}\geq3.
\]
 Moreover, $j_{n}(\lambda)-\lambda\leq2$. Hence
\begin{equation}
\frac{1}{\alpha_{j_{n}(\lambda)}}\left(\frac{2^{2-j_{n}(\lambda)}}{\tau_{j_{n}(\lambda)}}-\left(j_{n}(\lambda)-\lambda\right)\right)\geq\frac{1}{\alpha_{j_{n}(\lambda)}}.\label{eq:second term nu_n minimum when j>2}
\end{equation}
Next, assume $\lambda\leq1$. Then $j(\lambda)=2$, and we have 
\[
j_{n}(\lambda)-\lambda\leq2-\delta_{n},
\]
 whence
\begin{equation}
\frac{1}{\alpha_{j_{n}(\lambda)}}\left(\frac{2^{2-j_{n}(\lambda)}}{\tau_{j_{n}(\lambda)}}-\left(j_{n}(\lambda)-\lambda\right)\right)\geq\frac{\delta_{n}}{\alpha_{2}}.\label{eq:second term nu_n minimum when j=00003D2}
\end{equation}
\end{itemize}
We denote $\alpha=\max\{\alpha_{i}\}_{i=1}^{l}$ and $\nu_{n}=\frac{1}{2}\min\left\{ \frac{1}{\alpha},\ \frac{\delta_{n}}{\alpha},\ \frac{\eta_{n}}{3\alpha}\right\} $.
Then we conclude from (\ref{eq:nu definition}), (\ref{eq:first term in the minimum}),
(\ref{eq:second term nu_n minimum when j>2}) and (\ref{eq:second term nu_n minimum when j=00003D2})
that $\nu_{n}(\lambda)\geq\nu_{n}$. Importantly, we note that 
\[
\epsilon_{n}(\lambda)=\frac{1}{n^{\nu_{n}(\lambda)}}\leq\frac{1}{n^{\nu_{n}}}\overset{\text{def}}{=}\epsilon_{n}.
\]
Recall that $n^{-\eta_{n}}\to0$ and $n^{-\delta_{n}}\to0$, hence
$\epsilon_{n}\to0$. Now, by plugging in (\ref{eq:rho_n as phi_n})
and (\ref{eq:def of e_n(lambda)}) into (\ref{eq:estimate from vdc}),
we get that 
\begin{equation}
\frac{1}{n}\sum_{k=1}^{n}e\left(f_{n,\mathbf{x},\rho}(k)\right)\ll n^{T_{1,n}(\lambda)}+n^{T_{2,n}(\lambda)}+\frac{1}{n}+\epsilon_{n},\label{eq:estimate of vdc with T_i(=00005Clambda)}
\end{equation}
where
\[
T_{1,n}(\lambda)\overset{\text{def}}{=}\alpha_{j_{n}(\lambda)}\nu_{n}(\lambda)\left(2^{2-j_{n}(\lambda)}-\tau_{j_{n}(\lambda)}\right)-\tau_{j_{n}(\lambda)}\left(j_{n}(\lambda)-\lambda\right),
\]
\[
T_{2,n}(\lambda)\overset{\text{def}}{=}\tau_{j_{n}(\lambda)}\left(\alpha_{j_{n}(\lambda)}\nu_{n}(\lambda)+j_{n}(\lambda)-\lambda\right)-2^{2-j_{n}(\lambda)}.
\]
To finish the proof it remains to show that there exist sequences
$\left\{ T_{1,n}\right\} _{n=1}^{\infty},\ \left\{ T_{2,n}\right\} _{n=1}^{\infty}$
such that $T_{i,n}(\lambda)\leq T_{i,n}$ for $i=1,2$, and such that
$n^{T_{1,n}}\to0$ and $n^{T_{2,n}}\to0$.
\begin{itemize}
\item The term $T_{1,n}(\lambda)$: By definition of $\nu_{n}(\lambda)$
we have
\[
\nu_{n}(\lambda)\leq\frac{1}{2}\frac{j_{n}(\lambda)-\lambda}{\alpha_{j_{n}(\lambda)}\left(\frac{2^{2-j_{n}(\lambda)}}{\tau_{j_{n}(\lambda)}}-1\right)},
\]
hence, 
\begin{equation}
\begin{aligned}T_{1,n}(\lambda)\leq & \alpha_{j_{n}(\lambda)}\frac{1}{2}\frac{j_{n}(\lambda)-\lambda}{\alpha_{j_{n}(\lambda)}\left(\frac{2^{2-j_{n}(\lambda)}}{\tau_{j_{n}(\lambda)}}-1\right)}\left(2^{2-j_{n}(\lambda)}-\tau_{j_{n}(\lambda)}\right)-\tau_{j_{n}(\lambda)}\left(j_{n}(\lambda)-\lambda\right)\\
= & -\frac{1}{2}\tau_{j_{n}(\lambda)}\left(j_{n}(\lambda)-\lambda\right).
\end{aligned}
\label{eq:bound on T_1,n with j-phi}
\end{equation}
An inspection of (\ref{eq:j_n def}) shows that $\eta_{n}\leq j_{n}(\lambda)-\lambda,$
which combined with (\ref{eq:bound on T_1,n with j-phi}) gives
\[
T_{1,n}(\lambda)\leq-\frac{1}{2}\tau_{j_{n}(\lambda)}\eta_{n}.
\]
We define 
\[
T_{1,n}=-\frac{1}{2}\min\{\tau_{i}\}_{i=2}^{l}\eta_{n},
\]
then $T_{1,n}(\lambda)\leq T_{1,n}$ and as $n^{-\eta_{n}}\to0$,
we obtain that $n^{T_{1,n}}\to0$.
\item The term $T_{2,n}(\lambda)$: By our definition of $\nu_{n}(\lambda)$
we have
\[
\nu_{n}(\lambda)\leq\frac{1}{2}\frac{1}{\alpha_{j_{n}(\lambda)}}\left(\frac{2^{2-j_{n}(\lambda)}}{\tau_{j_{n}(\lambda)}}-\left(j_{n}(\lambda)-\lambda\right)\right),
\]
hence
\begin{equation}
\begin{aligned}T_{2,n}(\lambda)\leq & \tau_{j_{n}(\lambda)}\left(\alpha_{j_{n}(\lambda)}\frac{1}{2}\frac{1}{\alpha_{j_{n}(\lambda)}}\left(\frac{2^{2-j_{n}(\lambda)}}{\tau_{j_{n}(\lambda)}}-\left(j_{n}(\lambda)-\lambda\right)\right)+j_{n}(\lambda)-\lambda\right)-2^{2-j_{n}(\lambda)}\\
= & -\frac{\tau_{j_{n}(\lambda)}}{2}\left(\frac{2^{2-j_{n}(\lambda)}}{\tau_{j_{n}(\lambda)}}-\left(j_{n}(\lambda)-\lambda\right)\right).
\end{aligned}
\label{eq:T_2,n estimate}
\end{equation}
By (\ref{eq:second term nu_n minimum when j>2}) and (\ref{eq:second term nu_n minimum when j=00003D2})
we deduce from (\ref{eq:T_2,n estimate}) that 
\[
T_{2,n}(\lambda)\leq-\frac{1}{2}\tau_{j_{n}(\lambda)}\alpha_{j_{n}(\lambda)}\min\left\{ \frac{1}{\alpha_{j_{n}(\lambda)}},\frac{\delta_{n}}{\alpha_{2}}\right\} 
\]
Define 
\[
T_{2,n}=-\frac{1}{2}\min\{\tau_{i}\alpha_{i}\}_{i=2}^{l}\frac{\delta_{n}}{\max\{\alpha_{i}\}_{i=2}^{l}},
\]
then $T_{2,n}(\lambda)\leq T_{2,n}$ and since $n^{-\delta_{n}}\to0$,
we find that $n^{T_{2,n}}\to0$.
\end{itemize}
\end{proof}

\section{\label{sec:Proof-of-a.a. omega}Proof of theorem \ref{thm:almost all omega}}

For $\ef:\left[0,1\right]^{m}\times\R/\Z\to\R^{d}$, $\left\{ \mathbf{x}_{n}\right\} _{n=1}^{\infty}\subseteq\left[0,1\right]^{m},$
$\mathbf{h}\in\Z^{d}\smallsetminus\{\mathbf{0}\}$ and $\left\{ \rho_{n}\right\} _{n=1}^{\infty}\subseteq\R_{\geq0}$,
we define the function
\[
S_{n}(\omega)\overset{\text{def}}{=}\frac{1}{n}\sum_{k=1}^{n}e\left(\left\langle \mathbf{h},\rho_{n}\ef\left(\mathbf{x}_{n},\frac{k}{n}+\omega+\Z\right)\right\rangle \right),\ \omega\in\left[0,1\right].
\]
The following proposition is the main result of this section.
\begin{prop}
\label{prop:equivelent formulation to a.e,} Assume that $\ef:\left[0,1\right]^{m}\times\R/\Z\to\R^{d}$
is a family of $\RND$ analytic curves of order $\infty$. Let $\left\{ \mathbf{x}_{n}\right\} _{n=1}^{\infty}\subseteq\left[0,1\right]^{m},$
$\mathbf{h}\in\Z^{d}\smallsetminus\{\mathbf{0}\}$ and $\rho_{n}\to\infty$
be arbitrary. Then, for almost every $\omega\in\left[0,1\right]$,
\[
\lim_{n\to\infty}S_{n}(\omega)=0.
\]
\end{prop}

We now show that Proposition \ref{prop:equivelent formulation to a.e,}
implies Theorem \ref{thm:almost all omega}. Since a countable intersection
of full measure sets is of full measure, it follows from Proposition
\ref{prop:equivelent formulation to a.e,} that \textbf{$\lim_{n\to\infty}S_{n}(\omega)=0$}
\emph{for all} $\mathbf{h}\in\Z^{d}\smallsetminus\left\{ \mathbf{0}\right\} $,
for almost every $\omega\in\left[0,1\right]$. Hence Theorem \ref{thm:almost all omega}
follows by Weyl's equidistribution criterion.

To prove Proposition \ref{prop:equivelent formulation to a.e,} we
will use the Borel-Cantelli lemma with the following estimate of fourth
moments which we prove in Section \ref{subsec:Proof-of-Proposition 4.2}.
\begin{prop}
\label{prop:fourth moment estimate} Assume that $\ef:\left[0,1\right]^{m}\times\R/\Z\to\R^{d}$
is a family of $\RND$ analytic curves of order $\infty$ and fix
$\mathbf{h}\in\Z^{d}\smallsetminus\{\mathbf{0}\}$. For $\mathbf{x}\in\left[0,1\right]^{m}$,
$\rho>0$ and $n\in\N$ let\emph{
\[
S_{n}(\mathbf{x},\rho,\omega)=\frac{1}{n}\sum_{k=1}^{n}e\left(\left\langle \mathbf{h},\rho\ef\left(\mathbf{x},\frac{k}{n}+\omega+\Z\right)\right\rangle \right),\ \omega\in\left[0,1\right].
\]
} Then, there exists $\tau>0$ such that for $n\in\N$, $\mathbf{x}\in\left[0,1\right]^{m}$
and $\rho\geq n^{\tau}$, it holds
\begin{equation}
\int_{0}^{1}\left|S_{n}(\mathbf{x},\rho,\omega)\right|^{4}d\omega\ll\frac{1}{n^{2}},\label{eq:4th moment estimate}
\end{equation}
where the implied constant is independent of the parameters $n$,
$\mathbf{x}$ and $\rho$.
\end{prop}

We now explain how the statement of Proposition \ref{prop:fourth moment estimate}
implies Proposition \ref{prop:equivelent formulation to a.e,}.
\begin{proof}[Proof that Proposition \ref{prop:fourth moment estimate} implies
Proposition \ref{prop:equivelent formulation to a.e,}]
 Let $\ef:\left[0,1\right]^{m}\times\R/\Z\to\R^{d}$ be a family
of $\RND$ analytic curves of order $\infty$, $\left\{ \mathbf{x}_{n}\right\} _{n=1}^{\infty}\subseteq\left[0,1\right]^{m},$
$\mathbf{h}\in\Z^{d}\smallsetminus\{\mathbf{0}\}$ and $\rho_{n}\to\infty$
be arbitrary. Let $\tau$ be the exponent stated to exist in Proposition
\ref{prop:fourth moment estimate}.

We partition the sequence $\left\{ \rho_{n}\right\} _{n=1}^{\infty}$
into two subsequences, 
\[
S_{+}\overset{\text{def}}{=}\left\{ n\in\N\mid\rho_{n}>n^{\tau}\right\} ,\ S_{-}\overset{\text{def}}{=}\left\{ n\in\N\mid\rho_{n}\leq n^{\tau}\right\} .
\]
If $n\in S_{+},$ then 
\[
\left|\left\{ \omega\in[0,1]\mid\left|S_{n}(\omega)\right|\geq n^{-1/8}\right\} \right|=\left|\left\{ \omega\in[0,1]\mid\left|S_{n}(\omega)\right|^{4}\geq n^{-1/2}\right\} \right|\leq
\]
\[
\frac{\int_{0}^{1}\left|S_{n}(\omega)\right|^{4}d\omega}{n^{-1/2}}\underbrace{\ll}_{\text{Proposition \ref{prop:fourth moment estimate}}}\frac{1}{n^{3/2}},
\]
and this is summable. Therefore, if $S_{+}$ is infinite, the Borel-Cantelli
lemma shows
\[
\left|\left\{ \omega\in[0,1]\mid\exists N>0\text{ such that }\left|S_{n}(\omega)\right|<n^{-1/8},\ \forall n\in S_{+},\ n\geq N\right\} \right|=1,
\]
whence,
\[
\left|\left\{ \omega\in[0,1]\mid\lim_{S_{+}\ni n\to\infty}S_{n}(\omega)=0\right\} \right|=1.
\]

Next, assume without loss of generality that $S_{-}$ is infinite
(otherwise the proof is done by the above). We note that by the assumption
on $\ef$, for all $\omega\in[0,1]$, the family
\[
\phi_{\omega}(\mathbf{x},t)\overset{\text{def}}{=}\ef(\mathbf{x},t+\omega+\Z),\ \left(\mathbf{x},t\right)\in[0,1]^{m}\times[0,1],
\]
 is a $\RND$ analytic family of curves of order $\infty$. Hence,
Theorem \ref{thm:all omega} implies that for all $\omega\in[0,1]$
it holds
\[
\lim_{S_{-}\ni n\to\infty}S_{n}(\omega)=0.
\]
\end{proof}

\subsection{\label{subsec:Proof-of-Proposition 4.2}Proof of Proposition \ref{prop:fourth moment estimate}}

For the rest of the section we let $\ef:\left[0,1\right]^{m}\times\R/\Z\to\R^{d}$
be a family of $\RND$ analytic curves of order $\infty$ and we fix
$\mathbf{h}\in\Z^{d}\smallsetminus\{\mathbf{0}\}$.

Let us denote for $\mathbf{k}\in[n]^{4}$ and \textbf{$\mathbf{x}\in\left[0,1\right]^{m}$
}
\begin{equation}
f_{\mathbf{x},n,\mathbf{k}}(\omega)\overset{\text{def }}{=}\begin{array}{c}
\left\langle \mathbf{h},\sum_{i=1}^{4}(-1)^{i+1}\ef\left(\mathbf{x},\frac{k_{i}}{n}+\omega+\Z\right)\right\rangle ,\ \omega\in\left[0,1\right]\end{array},\label{eq:f_n_k}
\end{equation}
then using the above notation we get that
\begin{equation}
\begin{aligned}\int_{0}^{1}\left|S_{n}(\mathbf{x},\rho,\omega)\right|^{4}d\omega= & \int_{0}^{1}\left(S_{n}(\mathbf{x},\rho,\omega)\overline{S_{n}(\mathbf{x},\rho,\omega)}\right)^{2}d\omega\\
= & \frac{1}{n^{4}}\sum_{\mathbf{k}\in[n]^{4}}\int_{0}^{1}e\left(\rho f_{\mathbf{x},n,\mathbf{k}}(\omega)\right)d\omega
\end{aligned}
\label{eq:fourth moment as a sum of exponentials}
\end{equation}

Next, we recall the following well known estimate (see e.g. \cite[Chapter 8, Lemma 8.10]{Anaytic_num_theo})
which we will use below in the proof Lemma \ref{lem:estimate on exponential integral away from singular}.
\begin{lem}
\label{lem:estimate on exponential integral}There exists an absolute
constant $c>0$ such that $\forall f\in C^{2}(I)$ where $I=\left[a,b\right]$
that satisfy 
\[
\left|\frac{d^{2}}{d\omega^{2}}f(\omega)\right|\geq\lambda,\ \ \forall\omega\in I,
\]
it holds
\begin{equation}
\left|\int_{I}e(f(\omega))d\omega\right|\leq\frac{c}{\sqrt{\lambda}}.\label{eq:key estimate}
\end{equation}
\end{lem}

We consider the following analytic function 
\begin{equation}
\Phi(\mathbf{x},\mathbf{y},\omega)\overset{\text{def}}{=}\left\langle \mathbf{h},\sum_{i=1}^{4}(-1)^{i+1}\frac{\partial^{2}}{\partial\omega^{2}}\ef\left(\mathbf{x},y_{i}+\omega+\Z\right)\right\rangle ,\ \left(\mathbf{x},\mathbf{y},\omega\right)\in\left[0,1\right]^{m}\times\left[0,1\right]^{4}\times\left[0,1\right],\label{eq:def of big phi}
\end{equation}
which isn't constant (since $\ef$ is $\RND$ of order $\infty$),
and we note that it satisfies
\begin{equation}
\Phi\left(\mathbf{x},\frac{1}{n}\mathbf{k},\omega\right)=\frac{d^{2}}{d\omega^{2}}f_{\mathbf{x},n,\mathbf{k}}(\omega).\label{eq:relation of big Phi with f}
\end{equation}

We denote (as in Section \ref{sec:Prelimeneries})
\[
\Sigma=\left\{ \left(\mathbf{x},\mathbf{y}\right)\in\left[0,1\right]^{m}\times\left[0,1\right]^{4}\mid\Phi(\mathbf{x},\mathbf{y},t)=0,\ \forall t\in[0,1]\right\} ,
\]
and for $n\in\N$ (see (\ref{eq:definition of Sigma eps})),
\[
\Sigma_{n^{-2}}=\left\{ \left(\mathbf{x},\mathbf{y}\right)\in\left[0,1\right]^{m}\times\left[0,1\right]^{4}\mid\text{dist}(\left(\mathbf{x},\mathbf{y}\right),\Sigma)\geq n^{-2}\right\} .
\]

\begin{lem}
\label{lem:estimate on exponential integral away from singular} There
exists $\tau>0$ with the following property: for all $n\in\N,$ $\rho\geq n^{\tau}$,
\textbf{$\mathbf{x}\in\left[0,1\right]^{m}$} and $\mathbf{k}\in[n]^{4}$
such that $(\mathbf{x},\frac{1}{n}\mathbf{k)}\in\Sigma_{n^{-2}}$
it holds
\[
\int_{0}^{1}e\left(\rho f_{\mathbf{x},n,\mathbf{k}}(\omega)\right)d\omega\ll\frac{1}{n^{2}},
\]
where the implied constant is independent of the parameters $\mathbf{x},n$
and $\mathbf{k}$ \emph{(}depends on $\Phi$ only\emph{)}.
\end{lem}

\begin{proof}
We apply Proposition \ref{prop:uniformity on the lower bound and complement}
for the function $\Phi$ to get $\alpha=\alpha(\Phi),M=M(\Phi)>0$
such that for $(\mathbf{x},\frac{1}{n}\mathbf{k)}\in\Sigma_{n^{-2}}$
it holds (see (\ref{eq:relation of big Phi with f}))
\[
\mathcal{F}_{(\mathbf{x},\frac{\mathbf{k}}{n}),n^{-2\alpha}}\overset{\text{def}}{=}\left\{ \omega\in[0,1]\mid\left|\frac{d^{2}}{d\omega^{2}}f_{\mathbf{x},n,\mathbf{k}}(\omega)\right|\geq n^{-2\alpha}\right\} ,
\]
 is a union of at most $M$ intervals and 
\begin{equation}
1-\left|\mathcal{F}_{(\mathbf{x},\frac{\mathbf{k}}{n}),n^{-2\alpha}}\right|\ll\frac{1}{n^{2}},\label{eq:complement estimate}
\end{equation}
where the implied constant depends on $\Phi$ only. We define $\tau\overset{\text{def}}{=}4+2\alpha$
and we deduce that for $\rho\geq n^{\tau}$ it holds
\[
\left|\rho\frac{d^{2}}{d\omega^{2}}f_{\mathbf{x},n,\mathbf{k}}(\omega)\right|\geq n^{4},\ \forall\omega\in\mathcal{F}_{(\mathbf{x},\frac{\mathbf{k}}{n}),n^{-2\alpha}}.
\]
By applying Lemma \ref{lem:estimate on exponential integral} on each
of the intervals composing $\mathcal{F}_{(\mathbf{x},\frac{\mathbf{k}}{n}),n^{-2\alpha}}$
we obtain that
\begin{equation}
\int_{\mathcal{F}_{(\mathbf{x},\frac{\mathbf{k}}{n}),n^{-2\alpha}}}e\left(\rho f_{\mathbf{x},n,\mathbf{k}}(\omega)\right)d\omega\ll\frac{1}{n^{2}}.\label{eq:estimate}
\end{equation}
By (\ref{eq:complement estimate}) and (\ref{eq:estimate}) the proof
is complete.
\end{proof}
In order to show (\ref{eq:4th moment estimate}) it remains to prove
\begin{equation}
\#\left\{ \mathbf{k}\in[n]^{4}\mid\left(\mathbf{x},\frac{1}{n}\mathbf{k}\right)\in[0,1]^{m}\smallsetminus\Sigma_{n^{-2}}\right\} \ll n^{2},\label{eq:bound on rationals in the complement}
\end{equation}
uniformly in $\mathbf{x}\in[0,1]^{m}.$ Indeed, (\ref{eq:4th moment estimate})
will follow by applying the estimate of Lemma \ref{lem:estimate on exponential integral away from singular}
on the terms of (\ref{eq:fourth moment as a sum of exponentials})
for $\mathbf{k}\in\left[n\right]^{4}$ such that $\left(\mathbf{x},\frac{1}{n}\mathbf{k}\right)\in\Sigma_{n^{-2}}$
and by applying the trivial estimate on the terms of (\ref{eq:fourth moment as a sum of exponentials})
for $\mathbf{k}\in\left[n\right]^{4}$ such that $\left(\mathbf{x},\frac{1}{n}\mathbf{k}\right)\in[0,1]^{m}\smallsetminus\Sigma_{n^{-2}}$.

To prove (\ref{eq:bound on rationals in the complement}) the following
lemma is needed.
\begin{lem}
\label{lem:singular set} Consider the following parallelograms
\begin{equation}
H_{1}\df\left\{ \left(a,a,b,b\right)\mid a,b\in\left[0,1\right]\right\} ,\ \ H_{2}\df\left\{ \left(a,b,b,a\right)\mid a,b\in\left[0,1\right]\right\} ,\label{eq:def of H_i}
\end{equation}
and for $(l_{1},l_{2})\in\Z^{2}$ let 
\begin{equation}
\mathbf{l}\overset{\text{def}}{=}\left(l_{1}+l_{2},0,l_{2}-l_{1},0\right).\label{eq:def of l}
\end{equation}

Then there exists $j_{0}\in\N$ such that for all $\mathbf{x}\in\left[0,1\right]^{m}$
it holds
\begin{equation}
\left\{ \mathbf{y}\in\left[0,1\right]^{4}\mid\left(\mathbf{x},\mathbf{y}\right)\in\Sigma\right\} \subseteq\bigcup_{i=1}^{2}\bigcup_{\left\Vert \mathbf{l}\right\Vert _{\infty}\leq j_{0}}\left(H_{i}+\frac{1}{j_{0}}\mathbf{l}\right).\label{eq:union of j_0 translations of parallelograms}
\end{equation}
\end{lem}

\begin{proof}[Proof of Lemma \ref{lem:singular set}]
Recall that $\ef$ is a family of $\RND$ analytic curves of order
$\infty$, hence for any $\mathbf{x}\in\left[0,1\right]^{m}$, we
have that 
\[
\psi(\mathbf{x},\omega)\overset{\text{def}}{=}\frac{\partial^{2}}{\partial\omega^{2}}\left\langle \mathbf{h},\ef(\mathbf{x},\omega)\right\rangle ,\ \omega\in\R/\Z,
\]
 is a non-constant smooth function. Since the Fourier series of a
smooth function on $\R/\Z$ converges uniformly to the function (see
e.g. \cite{Eins_ward_functional}) and since $\psi(\mathbf{x},\cdot)$
is not constant, it follows that there exists $j\in\Z\smallsetminus\{0\}$
such that
\[
\widehat{\psi}(\mathbf{x},j)=\int_{\R/\Z}\psi(\mathbf{x},\omega)e(-j\omega)d\omega\neq0.
\]
Since $\psi(\mathbf{x},\cdot)$ is real, we have that $\widehat{\psi}(\mathbf{x},j)$
is the complex conjugate of $\widehat{\psi}(\mathbf{x},-j),$ so for
convenience we may assume that $j\in\N.$

We denote 
\[
j_{0}(\mathbf{x})=\min\{j\in\N\mid\widehat{\psi}(\mathbf{x},j)\neq0\},
\]
and in the following we show that $j_{0}(\mathbf{x})$ is bounded
in $\mathbf{x}\in[0,1]^{m}.$ Assume for contradiction that there
exists a sequence $\left\{ \mathbf{x}_{i}\right\} _{i=1}^{\infty}\subseteq\left[0,1\right]^{m}$
such that $j_{0}(\mathbf{x}_{i})\to\infty$. By compactness, we may
assume without loss of generality that $\mathbf{x}_{i}\to\mathbf{x}_{0}\in\left[0,1\right]^{m}$.
By continuity, it follows that $\hat{\psi}(\mathbf{x}_{0},j)=0$ for
all\textbf{ $j\in\N$}, which implies in turn that $\psi(\mathbf{x}_{0},\cdot)$
is constant, which is a contradiction. Whence we conclude that $j_{0}(\mathbf{x})$
is bounded.

For $\alpha\in\R/\Z$ we denote
\[
\tau_{\alpha}\psi(\mathbf{x},\omega)=\psi(\mathbf{x},\omega+\alpha),
\]
and observe that
\[
\widehat{\tau_{\alpha}\psi}(\mathbf{x},j)=e\left(j\alpha\right)\widehat{\psi}(\mathbf{x},j).
\]
We rewrite the function $\Phi$ (defined in (\ref{eq:def of big phi}))
as
\[
\Phi(\mathbf{x},\mathbf{y},\omega)=\sum_{i=1}^{4}(-1)^{i+1}\psi(\mathbf{x},\omega+y_{i}+\Z),
\]
and by the above we conclude that the $j_{0}(\mathbf{x})$'th Fourier
coefficient of $\Phi(\mathbf{x},\mathbf{y},\cdot)$ is 
\begin{equation}
\hat{\psi}(\mathbf{x},j_{0}(\mathbf{x}))\left(e\left(j_{0}(\mathbf{x})y_{1}\right)-e\left(j_{0}(\mathbf{x})y_{2}\right)+e\left(j_{0}(\mathbf{x})y_{3}\right)-e\left(j_{0}(\mathbf{x})y_{4}\right)\right).\label{eq:fourier coefficient of big phi}
\end{equation}
By (\ref{eq:fourier coefficient of big phi}) we deduce that
\[
\left\{ \mathbf{y}\in\left[0,1\right]^{4}\mid\left(\mathbf{x},\mathbf{y}\right)\in\Sigma\right\} \subseteq\left\{ \mathbf{y}\in[0,1]^{4}\mid e\left(j_{0}(\mathbf{x})y_{1}\right)-e\left(j_{0}(\mathbf{x})y_{2}\right)+e\left(j_{0}(\mathbf{x})y_{3}\right)-e\left(j_{0}(\mathbf{x})y_{4}\right)=0\right\} .
\]
Now, we recall the identity
\[
e(a)+e(b)=2\cos\left(\pi(a-b)\right)e\left(\frac{a+b}{2}\right),\!\ a,b\in\R,
\]
which yields
\[
e\left(j_{0}(\mathbf{x})y_{1}\right)-e\left(j_{0}(\mathbf{x})y_{2}\right)+e\left(j_{0}(\mathbf{x})y_{3}\right)-e\left(j_{0}(\mathbf{x})y_{4}\right)=0\iff
\]
\[
\begin{cases}
\cos\left(\pi j_{0}(\mathbf{x})(y_{1}-y_{3})\right)=\cos\left(\pi j_{0}(\mathbf{x})(y_{2}-y_{4})\right),\\
e\left(j_{0}(\mathbf{x})\frac{y_{1}+y_{3}}{2}\right)=e\left(j_{0}(\mathbf{x})\frac{y_{2}+y_{4}}{2}\right),
\end{cases}\iff
\]

there exist $l_{1},l_{2}\in\Z$ such that
\begin{equation}
\begin{cases}
y_{1}-y_{3}=y_{2}-y_{4}+\frac{2l_{1}}{j_{0}(\mathbf{x})},\text{ or }y_{1}-y_{3}=-\left(y_{2}-y_{4}\right)+\frac{2l_{1}}{j_{0}(\mathbf{x})},\\
y_{1}+y_{3}=y_{2}+y_{4}+\frac{2l_{2}}{j_{0}(\mathbf{x})}.
\end{cases}\label{eq:systems of linear equations determining H_i}
\end{equation}
For any fixed $l_{1},l_{2},l_{1}^{'},l_{2}^{'}\in\Z$, the solutions
to (\ref{eq:systems of linear equations determining H_i}) in\textbf{
}$\mathbf{y}\in[0,1]^{4}$ are included in 
\begin{equation}
\left\{ H_{1}+\left(\frac{l_{1}+l_{2}}{j_{0}(\mathbf{x})},0,\frac{l_{2}-l_{1}}{j_{0}\mathbf{(}\mathbf{x})},0\right)\right\} \bigcup\left\{ H_{2}+\left(\frac{l'_{1}+l_{2}^{'}}{j_{0}\mathbf{(}\mathbf{x})},0,\frac{l_{2}^{'}-l_{1}^{'}}{j_{0}\mathbf{(}\mathbf{x})},0\right)\right\} ,\label{eq:union of H_i}
\end{equation}
where $H_{i}$ is defined in (\ref{eq:def of H_i}). We also deduce
that $\left\Vert \mathbf{l}\right\Vert _{\infty},\left\Vert \mathbf{l}'\right\Vert _{\infty}\leq j_{0}(\mathbf{x})$,
since otherwise (\ref{eq:union of H_i}) will not intersect $\left[0,1\right]^{4}$.
Finally if $N_{0}\in\N$ is a bound for $j_{0}(\mathbf{x)}$, then
(\ref{eq:union of j_0 translations of parallelograms}) follows with
$j_{0}\overset{\text{def}}{=}\text{lcm}(1,...,N_{0}).$
\end{proof}
\begin{proof}[Proof that \emph{(\ref{eq:bound on rationals in the complement})}
holds]
We denote the finite union of parallelograms 
\begin{equation}
H\df\bigcup_{i=1}^{2}\bigcup_{\left\Vert \mathbf{l}\right\Vert _{\infty}\leq j_{0}}\left(H_{i}+\frac{1}{j_{0}}\mathbf{l}\right)\label{eq:union}
\end{equation}
satisfying by Lemma \ref{lem:singular set} that for all $\mathbf{x}\in[0,1]^{m}$
it holds
\begin{equation}
\left\{ \mathbf{y}\in\left[0,1\right]^{4}\mid\left(\mathbf{x},\mathbf{y}\right)\in\Sigma\right\} \subseteq H.\label{eq:union of hyperplanes contians fiber of sigma}
\end{equation}
 We claim that for all $\mathbf{x}\in\left[0,1\right]^{m}$ it holds
\begin{equation}
\left\{ \mathbf{y}\in\left[0,1\right]^{4}\mid\text{dist}\left((\mathbf{x},\mathbf{y}),\Sigma\right)<\frac{1}{n^{2}}\right\} \subseteq\left\{ \mathbf{y}\in\left[0,1\right]^{4}\mid\text{dist}\left(\mathbf{y},H\right)<\frac{1}{n^{2}}\right\} .\label{eq:contained in parallelograms}
\end{equation}
Indeed, let $\mathbf{y}_{0}\in[0,1]^{4}$ such that $\text{dist}\left((\mathbf{x},\mathbf{y}_{0}),\Sigma\right)<\frac{1}{n^{2}}$.
Then there exists $\left(\mathbf{a},\mathbf{b}\right)\in\Sigma$ such
that 
\[
\sqrt{\left\Vert \mathbf{x}-\mathbf{a}\right\Vert ^{2}+\left\Vert \mathbf{y}_{0}-\mathbf{b}\right\Vert ^{2}}<\frac{1}{n^{2}},
\]
which in turn implies that 
\[
\left\Vert \mathbf{y}_{0}-\mathbf{b}\right\Vert <\frac{1}{n^{2}}.
\]
By (\ref{eq:union of hyperplanes contians fiber of sigma}) we have
that $\mathbf{b}\in H$, which shows that $\mathbf{y}_{0}\in\left\{ \mathbf{y}\in\left[0,1\right]^{4}\mid\text{dist}\left(\mathbf{y},H\right)<\frac{1}{n^{2}}\right\} $.

We now conclude that

\begin{align*}
\#\left\{ \mathbf{k}\in[n]^{4}\mid\left(\mathbf{x},\frac{1}{n}\mathbf{k}\right)\in[0,1]^{m}\smallsetminus\Sigma_{n^{-2}}\right\} = & \#\left(\left\{ \mathbf{y}\in\left[0,1\right]^{4}\mid\text{dist}\left((\mathbf{x},\mathbf{y}),\Sigma\right)<\frac{1}{n^{2}}\right\} \cap\frac{1}{n}\Z^{4}\right)\\
\underbrace{\leq}_{\eqref{eq:contained in parallelograms}} & \#\left(\left\{ \mathbf{y}\in\left[0,1\right]^{4}\mid\text{dist}\left(\mathbf{y},H\right)<\frac{1}{n^{2}}\right\} \cap\frac{1}{n}\Z^{4}\right)\\
\leq\sum_{i=1}^{2}\sum_{\left\Vert \mathbf{l}\right\Vert _{\infty}\leq j_{0}} & \#\left(\left\{ \mathbf{y}\in\left[0,1\right]^{4}\mid\text{dist}\left(\mathbf{y},H_{i}+\frac{1}{j_{0}}\mathbf{l}\right)<\frac{1}{n^{2}}\right\} \cap\frac{1}{n}\Z^{4}\right)
\end{align*}

Therefore, to prove the estimate $\#\left\{ \mathbf{k}\in[n]^{4}\mid\left(\mathbf{x},\frac{1}{n}\mathbf{k}\right)\in[0,1]^{m}\smallsetminus\Sigma_{n^{-2}}\right\} \ll n^{2}$
uniformly in $\mathbf{x}\in[0,1]^{m},$ it is sufficient to verify
that 
\begin{equation}
\#\left(\left\{ \mathbf{y}\in\left[0,1\right]^{4}\mid\text{dist}\left(\mathbf{y},H_{i}+\mathbf{v}\right)<\frac{1}{n^{2}}\right\} \cap\frac{1}{n}\Z^{4}\right)\ll n^{2},\label{eq:1/nZ^4 points near pallalgram}
\end{equation}
uniformly in $i\in\{1,2\}$ and $\mathbf{v}\in[0,1]^{4}$. To prove
(\ref{eq:1/nZ^4 points near pallalgram}) we note that for $n\in\N$
there is a cover$^{\nameref{enu:-footnote cover}}$\footnote{\begin{elabeling}{00.00.0000}
\item [{(1)\label{enu:-footnote cover}}] Such a sequence of covers is
obtained by partitioning the parallelogram $(H_{i}+\mathbf{v})\cap[0,1]^{4}$
into $\ll n^{2}$ parallelograms $P_{i}(n)$ of side length $\ll\frac{1}{10}\frac{1}{n}$
and defining $S_{i}(n)\df P_{i}(n)+D_{\frac{1}{10}\frac{1}{n}}$,
where $D_{r}\df\{\alpha_{1}\mathbf{u}_{1}+\alpha_{2}\mathbf{u}_{2}\mid\sqrt{\alpha_{1}^{2}+\alpha_{2}^{2}}\leq r\}$
for an orthonormal basis $\{\mathbf{u}_{1},\mathbf{u}_{2}\}$ of $H_{i}^{\perp}$
the plane orthogonal to $H_{i}$.
\end{elabeling}
} 
\[
\left\{ \mathbf{y}\in\left[0,1\right]^{4}\mid\text{dist}\left(\mathbf{y},H_{i}+\mathbf{v}\right)<\frac{1}{n^{2}}\right\} \subseteq S_{1}(n)\cup...\cup S_{m(n)}(n),
\]
where $m(n)\ll n^{2}$ uniformly in $\mathbf{v}$, such that
\begin{equation}
\text{Euclidean diameter of }S_{j}(n)\leq\frac{1}{2}\cdot\frac{1}{n}.\label{eq:diameter}
\end{equation}
By (\ref{eq:diameter}), each set $S_{j}(n)$ can contain at most
one rational vector $\frac{1}{n}\mathbf{k}$ where $\mathbf{k}\in\left[n\right]^{4}$,
which shows (\ref{eq:1/nZ^4 points near pallalgram}).
\end{proof}

\section{\label{sec:On-the-growth}Counter examples}

Our main tool to prove Theorems \ref{lem:failure for rnd of order k}
and \ref{lem:failure for RND} is the well known Dirichlet's simultaneous
approximation theorem, which we recall now.
\begin{thm}
For any $M\in\N$ and $\mathbf{x}\in\R^{N}$, there exists $\mathbf{p}\in\Z^{N}$
and $q\in\{1,...,M\}$ such that 
\[
\left\Vert q\mathbf{x}-\mathbf{p}\right\Vert _{\infty}\leq\frac{1}{M^{1/N}}.
\]
\end{thm}

\begin{proof}[Proof of Theorem \ref{lem:failure for rnd of order k}]
 Assume that $\gamma$ is $\RND$ analytic curve of order $\kappa\in\N$.
Then, there exists $\mathbf{h}\in\Z^{d}\smallsetminus\{\mathbf{0}\}$
such that 
\begin{equation}
\left\langle \mathbf{h},\gamma^{(\kappa+1)}(t)\right\rangle =0,\ \forall t\in[0,1].\label{eq:gamma contained in rat hyperplane}
\end{equation}
Then (\ref{eq:gamma contained in rat hyperplane}) implies that $\left\langle \mathbf{h},\gamma(t)\right\rangle $
is a polynomial of degree $\kappa$, say
\[
\left\langle \mathbf{h},\gamma(t)\right\rangle =a_{\kappa}t^{\kappa}+..+a_{0},\ t\in\left[0,1\right].
\]
By Dirichlet's theorem for each $n\in\N$, $M=n^{\kappa^{2}+1}$ and
the points 
\[
\left(a_{\kappa},a_{\kappa-1},...,a_{1}\right)\in\R^{\kappa},\ 
\]
we find that there exists $\tilde{\rho}_{n}\in\N$ such that $\tilde{\rho}_{n}\leq n^{\kappa^{2}+1},$
and $\mathbf{p}=\left(p_{1},..,p_{\kappa}\right)$ where \textbf{$p_{i}\in\Z$},
such that 
\begin{equation}
\left|\tilde{\rho}_{n}a_{i}-p_{i}\right|\leq\frac{1}{n^{\kappa+\frac{1}{\kappa}}},\ \forall i\in\left\{ 1,..,\kappa\right\} .\label{eq:bound for polynomial from inner produc for k+1 deg curve}
\end{equation}
We observe that 
\begin{equation}
\left\langle \mathbf{h},\ \left(n^{\kappa}\tilde{\rho}_{n}\right)\gamma\left(\frac{j}{n}\right)\right\rangle =j^{\kappa}\tilde{\rho}_{n}a_{\kappa}+nj^{\kappa-1}\tilde{\rho}_{n}a_{\kappa-1}+..+n^{\kappa}\tilde{\rho}_{n}a_{0},\label{eq:inner_prod after dilation rho_n_timesN^k}
\end{equation}
and that for all $i\in\left\{ 0,1,..,\kappa-1\right\} ,$ we have
\begin{equation}
\left|n^{i}j^{\kappa-i}\tilde{\rho}_{n}a_{\kappa-i}-n^{i}j^{\kappa-i}p_{\kappa-i}\right|=n^{i}j^{\kappa-i}\left|\tilde{\rho}_{n}a_{\kappa-i}-p_{\kappa-i}\right|\underbrace{\leq}_{\eqref{eq:bound for polynomial from inner produc for k+1 deg curve}}\frac{1}{n^{\frac{1}{\kappa}}}.\label{eq:difference of each term of the dilations from integers}
\end{equation}
By (\ref{eq:inner_prod after dilation rho_n_timesN^k}) and (\ref{eq:difference of each term of the dilations from integers})
we deduce for each $j\in\left\{ 1,..,n\right\} $ that
\[
\left|\left(\left\langle \mathbf{h},\left(n^{\kappa}\tilde{\rho}_{n}\right)\gamma\left(\frac{j}{n}\right)\right\rangle -n^{\kappa}\tilde{\rho}_{n}a_{0}\right)-\sum_{i=1}^{n}n^{i}j^{\kappa-i}p_{\kappa-i}\right|\leq\frac{\kappa}{n^{\frac{1}{\kappa}}}.
\]
We denote 
\[
\delta_{n,j}\overset{\text{def}}{=}\left(\left\langle \mathbf{h},\left(n^{\kappa}\tilde{\rho}_{n}\right)\gamma\left(\frac{j}{n}\right)\right\rangle -n^{\kappa}\tilde{\rho}_{n}a_{0}\right)-\sum_{i=1}^{n}n^{i}j^{\kappa-i}p_{\kappa-i},
\]
and we conclude that $\lim_{n\to\infty}\delta_{n,j}=0$ uniformly
in $j$, which implies in turn that
\begin{equation}
\sum_{j=1}^{n}e\left\langle \mathbf{h},n^{\kappa}\tilde{\rho}_{n}\gamma\left(\frac{j}{n}\right)\right\rangle =e\left(n^{\kappa}\tilde{\rho}_{n}a_{0}\right)\sum_{j=1}^{n}e(\delta_{n,j})\asymp n.\label{eq:doesnt decay}
\end{equation}
 We define the sequence $\rho_{n}\overset{\text{def}}{=}n^{\kappa}\tilde{\rho}_{n},$
and we note that $n^{\kappa}\leq\rho_{n}\leq n^{\left(\kappa+1\right)^{2}}$.
Finally, by (\ref{eq:doesnt decay}), we get that $\left\{ \frac{1}{n}\sum_{j=1}^{n}e\left(\left\langle \mathbf{h},\rho_{n}\gamma\left(\frac{j}{n}\right)\right\rangle \right)\right\} _{n=1}^{\infty}$
will not converge to zero.
\end{proof}
\begin{proof}[Proof of Theorem \ref{lem:failure for RND}]
 Let $\phi:\left[0,1\right]^{m}\times\left[0,1\right]\to\R^{d}$
be a family curves, $\left\{ \mathbf{x}_{n}\right\} _{n=1}^{\infty}\subseteq\left[0,1\right]^{m}$.
Then by Dirichlet's theorem for each $n\in\N$, $M=3^{dn}$ and the
points 
\[
\left(\log n\right)\left(\phi(\mathbf{x}_{n},1/n),\ \phi(\mathbf{x}_{n},2/n),...,\phi(\mathbf{x}_{n},1)\right)\in\R^{dn},\ 
\]
there exists $\tilde{\rho}_{n}\in\N$ such that $\tilde{\rho}_{n}\leq3^{dn},$
and $\mathbf{p}=\left(\mathbf{p}_{1},..,\mathbf{p}_{n}\right)$ where
\textbf{$\mathbf{p}_{i}\in\Z^{d}$}, such that 
\begin{equation}
\left|\tilde{\rho}_{n}\left(\left(\log n\right)\phi\left(\mathbf{x}_{n},\frac{j}{n}\right)\right)-\mathbf{p}_{j}\right|_{\infty}\leq\frac{1}{3},\ \forall j\in\left\{ 1,..,n\right\} .\label{eq:dirichlet thm for the curves}
\end{equation}
Denote $\rho_{n}\overset{\text{def}}{=}\tilde{\rho}_{n}\log n$, and
note that $\rho_{n}\to\infty$ and $\rho_{n}\leq\left(3.5\right)^{dn}$
for all large enough $n$. Finally, let $\gamma_{n}\df\rho_{n}\phi(\mathbf{x}_{n},\cdot)$,
and observe by (\ref{eq:dirichlet thm for the curves}) that $\left\{ \pi(\gamma_{n}(j/n))\right\} _{j=1}^{n}$
is contained in a strict subset of $\R^{d}/\Z^{d}$ of measure $\left(\frac{2}{3}\right)^{d}$
for all $n\in\N$, so that $\left\{ \mu_{n}\right\} _{n=1}^{\infty}$
will not equidistribute.
\end{proof}

\appendix

\section{\label{sec:Basic-notions-in o-minim}Basic notions in the theory
of o-minimal structures}

In order to make our paper self contained we discuss some basic notions
in the theory of o-minimal structures which we use in the proof of
Proposition \ref{prop:uniformity on the lower bound and complement}.
For more details on o-minimal structures we refer to the book \cite{tame_topo_and_o_min_structures}.
\begin{defn}
\label{def:structure-}\emph{A} \emph{structure} $\A$ \emph{on the
real field$^{\nameref{enu:footnote o-min}}$}\footnote{\begin{elabeling}{00.00.0000}
\item [{(2)\emph{\label{enu:footnote o-min}}}] The definition given here
is equivalent to the definition given in \cite[Section 2]{Dries_Miller_Geometric_Categories}.
We note that the notion of a structure is more general, see e.g. \cite{tame_topo_and_o_min_structures}.
\end{elabeling}
} is a sequence $\A\df\left(\A_{d}\right)_{d=1}^{\infty}$ where $\mathcal{A}_{d}$
is a subset of the power set of $\R^{d}$, satisfying the following
requirements for all $d,m\in\N$:
\begin{enumerate}
\item \label{enu:boolean}$\A_{d}$ is a Boolean algebra, namely $\emptyset\in\A_{d}$
and $\A_{d}$ is closed under the operation of taking a complement
or by performing a finite union.
\item \label{enu:the-diagonals-}The diagonals $\Delta_{i,j}\df\left\{ \left(x_{1},..,x_{d}\right)\mid x_{i}=x_{j}\right\} $
for all $1\leq i<j\leq d$ belong to $\A_{d}$.
\item \label{enu:prod}For all $A\in\A_{d}$ and $B\in\A_{m}$ it holds
that $A\times B\in\A_{d+m}$ and $B\times A\in\A_{d+m}$.
\item \label{enu:proj}For all $A\in\A_{d+m}$ it holds that $\pi_{d+m,d}(A)\in\A_{d}$,
where $\pi_{d+m,d}:\R^{d+m}\to\R^{d}$ denotes the projection to the
first $d$ coordinates.
\item \label{enu:addition and multiplication}The graphs of addition and
multiplication are in $\A_{3}$.
\end{enumerate}
\end{defn}

\begin{defn}
Let $\A$ \emph{be a structure} \emph{on the real field}. \emph{If
any set that belong to $\A_{1}$ is comprised of a finite union of
points or intervals, then we say that $\A$ is an o-minimal structure.}
\end{defn}

\begin{defn}
Fix a structure on the real field\emph{ $\A=\left(\A_{d}\right)_{d=1}^{\infty}$.}

We say that $A\subseteq\R^{d}$ is definable in $\A$ if $A\in\A_{d}$,
we say that $f:A'\to\R^{m}$ for $A\in\A_{d}$ is definable in $\A$
if the graph of $f$ is in $\A_{d+m}$, and we say that a constant
$c\in\R$ is definable in $\A$ if $\{c\}\in A_{1}$.
\end{defn}

\begin{lem}
\label{lem:basic properties of structure}The following hold for any
structure $\A=\left(\A_{d}\right)_{d=1}^{\infty}$ on the real field\emph{:}
\begin{enumerate}
\item Let $\sigma:\{1,..,d\}\to\{1,..,d\}$ be a permutation. Then $A\in\A_{d}$
is definable if and only if 
\[
A^{\sigma}\df\{(x_{1},..,x_{d})\in\R^{d}\mid(x_{\sigma(1)},..,x_{\sigma(d)})\in A\}
\]
is definable.
\item \label{enu:For-any-structure}For a definable function $f:\R^{d}\to\R$
and a definable set $A\in\A_{d}$, it holds that 
\[
\left\{ \mathbf{x}\in A\mid f(\mathbf{x})=0\right\} ,\ \left\{ \mathbf{x}\in A\mid f(\mathbf{x})>0\right\} 
\]
are definable.
\item If $f:\R^{m}\to\R$ is definable, then the restriction of $f$ to
a definable set $A\in\A_{m}$ is a definable function.
\item If $f:\R^{m}\to\R^{n}$ and $g:\R^{n}\to\R^{k}$ are definable, then
so is the composition $g\circ f$.
\item If $f,g:\R^{m}\to\R$ are definable, then $f+g,$ $f-g,$ $f\cdot g$
are definable, and $f/g$ defined in the domain $\{\mathbf{x}\in\R^{m}\mid g(\mathbf{x})\neq0\}$
is definable.
\item \label{enu:fix parameters}If $c_{1},..,c_{n}\in\R$ are definable
constants and $f:\R^{m+n}\to\R$ is definable, then so is $g(x_{1},...,x_{m})\df f(x_{1},...,x_{m},c_{1},..,c_{n})$.
\end{enumerate}
\end{lem}

\begin{proof}
The proof of the lemma relies only on Definition \ref{def:structure-}
and follows in a rather straightforward manner. In the following we
only prove (\ref{enu:For-any-structure}), and leave the rest for
the reader.

We first show that $0$ and $\R_{>0}$ are definable. We have
\[
\{(x,x,2x)\mid x\in\R\}=\{(x,y,x+y)\mid x,y\in\R\}\cap\Delta_{1,2},
\]
and 
\[
\{0\}=\pi_{3,1}\left(\{(x,x,2x)\mid x\in\R\}\cap\Delta_{2,3}\right),
\]
which shows that $0$ is definable. To show that $\R_{>0}$ is definable,
we observe that 
\[
\{(x,x,x^{2})\mid x\in\R\}=\{(x,y,xy)\mid x\in\R\}\cap\Delta_{1,2},
\]
and
\[
\R_{>0}=\pi_{4,1}\left(\left(\R\times\{(x,x,x^{2})\mid x\in\R\}\right)\cap\Delta_{1,4}\right)\smallsetminus\{0\}.
\]

Let $f:\R^{d}\to\R$ be definable, let $A\in\A_{d}$, and denote the
graph of $f$ by $\Gamma(f)$. Then 
\[
\left\{ \mathbf{x}\in A\mid f(\mathbf{x})>0\right\} =\pi_{d+1,d}\left(\Gamma(f)\cap\left(A\times\R_{>0}\right)\right),
\]
and 
\[
\left\{ \mathbf{x}\in A\mid f(\mathbf{x})=0\right\} =\pi_{d+1,d}\left(\Gamma(f)\cap\left(A\times\{0\}\right)\right).
\]
are definable.
\end{proof}

\subsection{Formulae}

A convenient way to describe sets in a structure on the real field
is to use formulae which are defined as follows.

\begin{defn}
An atomic formula in a structure on the real field $\A=(\A_{d})_{d=1}^{\infty}$
is any of the following expressions
\begin{itemize}
\item $(x_{1},...,x_{d})\in A$ is an atomic formula, where $A\in\A_{d}$,
\item $f(x_{1},..,x_{d})>0$ and $f(x_{1},..,x_{d})=0$ are atomic formulae,
where $f:A'\to\R$ is a definable function.
\end{itemize}
A formula is defined inductively by the following
\begin{itemize}
\item An atomic formula is a formula
\item If $\phi(x_{1},...,x_{d})$ and $\psi(x_{1},...,x_{d})$ are formulae,
then $\phi\land\psi$, $\phi\lor\psi$, $\lnot\phi$ and $\phi\Rightarrow\psi$
are formulae.
\item If $A\in\A_{m}$, and $\phi(x_{1},..,x_{n},y_{1},..,y_{m})$ is a
formula, then \\
$\exists(y_{1},..,y_{m})\in A(\phi(x_{1},..,x_{n},y_{1},...,y_{m}))$
and $\forall(y_{1},..,y_{m})\in A(\phi(x_{1},..,x_{n},y_{1},...,y_{m}))$
are formulae.
\end{itemize}
\end{defn}

\subsubsection{Description of sets via formulae}

We fix a structure on the real field $\A=(\A_{d})_{d=1}^{\infty}$,
and in the following we denote by $A$ a definable set and we denote
by $f$ a definable function $f:A'\to\R$ for $A'\in\A_{d}$.

Atomic formulae define sets by the following obvious manner:

\[
\begin{array}{ccc}
(x_{1},...,x_{d})\in A & \text{defines} & A,\\
f(x_{1},..,x_{d})>0 & \text{defines} & \{\mathbf{x}\in A'\mid f(\mathbf{x})>0\},\\
f(x_{1},..,x_{d})=0 & \text{defines} & \{\mathbf{x}\in A'\mid f(\mathbf{x})=0\}.
\end{array}
\]
Now assume that the formulae $\phi(\mathbf{x},\mathbf{y})$ and $\text{\ensuremath{\psi}(}\mathbf{x},\mathbf{y})$
define $\Phi\subseteq\R^{n}\times\R^{m}$ and $\Psi\subseteq\R^{n}\times\R^{m}$
correspondingly. Then the following formulae define sets by interpreting
logical symbols with Boolean operations and projections in the following
standard way:
\[
\begin{array}{ccc}
\phi(\mathbf{x},\mathbf{y})\land\psi(\mathbf{x},\mathbf{y}) & \text{defines} & \Phi\cap\Psi,\\
\phi(\mathbf{x},\mathbf{y})\lor\psi(\mathbf{x},\mathbf{y}) & \text{defines} & \Phi\cup\Psi,\\
\lnot\phi(\mathbf{x},\mathbf{y}) & \text{defines} & \R^{n+m}\smallsetminus\Phi,\\
\exists\mathbf{y\in}A(\phi(\mathbf{x},\mathbf{y})) & \text{defines} & \pi_{m+n,n}(\Phi\cap(\R^{n}\times A)),\\
\forall\mathbf{y}\in A(\phi(\mathbf{x},\mathbf{y})) & \text{defines the same set as} & \lnot\ \exists\mathbf{y}\in A(\lnot\phi(\mathbf{x},\mathbf{y})),\\
\phi(\mathbf{x},\mathbf{y})\Rightarrow\psi(\mathbf{x},\mathbf{y}) & \text{defines the same set as} & \lnot\phi(\mathbf{x},\mathbf{y})\lor\psi(\mathbf{x},\mathbf{y}).
\end{array}
\]

By Definition \ref{def:structure-} and Lemma \ref{lem:basic properties of structure}
we deduce that formulae yield definable sets.

\subsection{\label{subsec:The-o-minimal-structure R_an}The o-minimal structure
of restricted analytic functions}

We obtain a partial order on all structures expanding the real field
by declaring that for two structures $\A$ and $\A^{'}$ it holds
that $\A'\leq\A$ if and only if $\A'_{d}\subseteq\A_{d}$ for all
$d\in\N$.

In this paper we are only interested in the structure $\R_{an}$ known
as the field of real numbers with restricted analytic functions, which
is defined to be the smallest structure on the real field in which
all real numbers are definable and all functions $f:\R^{d}\to\R$
which are real analytic in $\left[0,1\right]^{d}$ and vanish in $\R^{d}\smallsetminus\left[0,1\right]^{d}$
(namely, $f$ is the restriction of an analytic function in a neighborhood
of $\left[0,1\right]^{d}$) are definable.

In \cite{Gen_of_tarski_seindberg} it was shown that $\R_{an}$ is
an o-minimal structure.
\begin{lem}
\label{lem:A definability}Let $F:\left[0,1\right]^{N}\times\left[0,1\right]\to\R$
be an analytic function. Then the set $A$ defined in Section \ref{sec:Prelimeneries}\textbf{
}is definable in $\R_{an}$.
\end{lem}

\begin{proof}
In order to prove the claim we give an explicit formula which defines
\[
A\df\left\{ \left(\epsilon,\delta\right)\in(0,1]^{2}\mid\forall\mathbf{x}\in\Sigma_{\epsilon},\ \forall\xi\in\left[0,1\right]\text{ it holds that }(\xi-\frac{\epsilon}{2},\xi+\frac{\epsilon}{2})\cap\mathcal{F}_{\mathbf{x},\delta}\neq\emptyset\right\} ,
\]
where\textbf{ }$\Sigma_{\epsilon}\df\left\{ \mathbf{x}\in[0,1]^{N}\mid d(\mathbf{x},\Sigma)\geq\epsilon\right\} $
and $d(\cdot,\Sigma):\R^{N}\to\R$ is the Euclidean distance function
from the closed set $\Sigma\df\left\{ \mathbf{x}\in[0,1]^{N}\mid F(\mathbf{x},\cdot)\equiv0\right\} $.
\\
We now show that $d(\cdot,\Sigma)$ is definable in $\R_{an}$. Indeed,
we note that the set $\Sigma$ is definable since it is defined by
\[
\forall t\in[0,1](F(\mathbf{x},t)=0)).
\]
Therefore we get that the following is a formula in $\R_{an}$ 
\[
(\mathbf{x},y)\in\R^{N}\times\R_{\geq0}\land\left(\forall\epsilon>0,\exists\xi\in\Sigma(y^{2}\leq\left\Vert \mathbf{x}-\xi\right\Vert ^{2}\leq(y+\epsilon)^{2})\right),
\]
which defines the graph of $d(\cdot,\Sigma):\R^{N}\to\R$, showing
that $d(\cdot,\Sigma)$ is definable.

We may now deduce that the following is a formula in $\R_{an}$
\begin{align*}
\phi(\epsilon,\delta,\mathbf{x},\xi)\df(\mathbf{x},\xi,(\delta,\epsilon))\in[0,1]^{N}\times[0,1]\times(0,1]^{2}\land\left(\ d(\mathbf{x},\Sigma)-\epsilon>0\right.\\
\Rightarrow\exists t\in[0,1]((\xi-\frac{\epsilon}{2}\leq t\leq\xi+\frac{\epsilon}{2})\land F(\mathbf{x},t)^{2}-\delta^{2}\geq0)\ ).
\end{align*}
Since $A$ is defined by 
\begin{align*}
\left(\epsilon,\delta\right)\in(0,1]^{2}\land\forall(\mathbf{x},\xi)\in[0,1]^{N}\times[0,1] & \phi(\epsilon,\delta,\mathbf{x},\xi)
\end{align*}
we get that $A$ is definable. 
\end{proof}

\subsection{Needed results from the theory of o-minimal structures}

The following two properties hold for an arbitrary o-minimal structure
$\A=(\A_{d})_{d=1}^{\infty}$.
\begin{thm}
\label{thm:uniform bound on fibers}\cite[Chapter 3, Corollary 3.6]{tame_topo_and_o_min_structures}
Let $A\in\A_{m+n}.$ Then there exists $N\in\N$ such for all $\mathbf{x}\in\R^{m}$
the set $A_{\mathbf{x}}\df\{\mathbf{y}\in\R^{n}\mid(\mathbf{x},\mathbf{y})\in A\}$
has at most $N$ connected components.
\end{thm}

\begin{thm}
\label{thm:choice function}\cite[Chapter 6, Proposition 1.2]{tame_topo_and_o_min_structures}
Let $A\in\A_{m+n}$. Then there is a definable map $f:\pi_{m+n,n}(A)\to\R^{m+n}$
such that the graph of $f$ is contained in $A$.
\end{thm}

An important property possessed by $\R_{an}$ is the property of polynomial
boundedness (see \cite{Gen_of_tarski_seindberg})\textbf{ }which we
define below.
\begin{defn}
\label{def:polynomially bounded} A structure $\A=(\A_{d})_{d=1}^{\infty}$
is \emph{polynomially bounded }if for every definable $f:\R\to\R$
there exists $m\in\N$ such that $f(t)=O(t^{m})$ as $t\to\infty$.
\end{defn}

We have the following straightforward corollary (for a more general
version see \cite[4.13]{Dries_Miller_Geometric_Categories})
\begin{cor}
\label{cor:polynomially bounded from below}Let $f:(0,a)\to\R$ be
strictly positive and bounded function definable in $\R_{an}$. Then
there exist $0<c,\kappa$ and $0<\epsilon<a$ such that $ct^{\kappa}\leq f(t)$
for all $t\in(0,\epsilon)$.
\end{cor}

\begin{proof}
For convenience we extend $f$ to the real line by defining $f(t)=1$
for all $t\notin(0,a)$ (which also gives a function definable in
$\R_{an}$). Since $f$ is non-vanishing, we may consider the function
$g:\R\to\R$ defined by $g(s)=\frac{1}{f(1/s)},$ (which is definable
since $1/s$ for $s\neq0$ is definable, and composition of definable
functions gives a definable function).

Since $\R_{an}$ is polynomially bounded, there is $c'>0$ such that
$\frac{1}{f(1/s)}\leq c's^{m}$ for all $s$ large enough. Therefore,
for any $t$ close enough to $0$ from the right we have that 
\[
f(t)\geq\frac{1}{c'}t^{m}.
\]
\end{proof}
\bibliographystyle{alpha}
\bibliography{revision_final}

\end{document}